\documentclass[10pt]{article}
\usepackage{amsmath,amssymb,amsfonts,amsthm} % Typical maths resource packages
\usepackage{mathtools}
\usepackage{stmaryrd} %Widerspruchspfeil
\usepackage{color}
\usepackage{calc}
\usepackage[version=3,arrows=pgf-filled]{mhchem}
\usepackage{csquotes}
\usepackage{colortbl}
\usepackage{extarrows}
\usepackage[ruled]{algorithm2e} %for algorithms
\usepackage{mathrsfs} %fancy caligraphy
\usepackage{import} % different paths for input calls
\usepackage{enumitem}
\usepackage{appendix}
\usepackage{fullpage}
\usepackage{pgfplots}
\usepackage{thmtools}   % for restatable
\usepackage{dsfont} % for \mathds{1}
\usepackage{tikz}
\usepackage{fullpage}
\usepackage{longtable} % to allow for pagebreaks in notation table
\usepackage{booktabs} % to add top/bottom row for table
\allowdisplaybreaks % allow equation break across pages

\usepackage{hyperref}
\hypersetup{
colorlinks   = true, %Colours links instead of boxes
urlcolor     = black, %Colour for external hyperlinks
linkcolor    = blue, %Colour of internal links
citecolor   = blue %Colour of citations, could be ``red''
}
\usepackage{colonequals} %to use the := symbol
\usepackage{soul} % strikethrough \st{}

\usepackage{dsfont}

\usepackage{subfig} 

\usepackage{pgfplots}
\pgfplotsset{compat=1.7}

\usepackage{nccmath}

\mathchardef\mhyphen="2D 

\theoremstyle{definition}
\newtheorem{defi}{Definition}[section]

\theoremstyle{remark}
\newtheorem{rem}[defi]{Remark}
\usepackage{etoolbox} % Required for the  & ext command for adding QED after remarks
\AtEndEnvironment{rem}{\qed}
\AtEndEnvironment{example}{\qed}
\newtheorem{example}[defi]{Example}

\theoremstyle{plain}
\newtheorem{theorem}[defi]{Theorem}
\newtheorem{lem}[defi]{Lemma}
\newtheorem{cor}[defi]{Corollary}
\newtheorem{prop}[defi]{Proposition}

\newcommand{\N}{\ensuremath{\mathbb{N}}}
\newcommand{\R}{\ensuremath{\mathbb{R}}}

\newcommand{\Z}{\ensuremath{\mathbb{Z}}}

\newcommand{\1}{\mathbf{1}}

\newcommand{\e}{\varepsilon}

\newlength{\leftstackrelawd}
\newlength{\leftstackrelbwd}
\def\leftstackrel#1#2{\settowidth{\leftstackrelawd}%
{${{}^{#1}}$}\settowidth{\leftstackrelbwd}{$#2$}%
\addtolength{\leftstackrelawd}{-\leftstackrelbwd}%
\leavevmode\ifthenelse{\lengthtest{\leftstackrelawd>0pt}}%
{\kern-.5\leftstackrelawd}{}\mathrel{\mathop{#2}\limits^{#1}}}

\def\RelEnt{\mathscr{H}}
\def\RF{\mathscr{R}}

\def\Ent{\mathrm{Ent}}

\newcommand{\calE}{\mathscr{E}}

\def\X{\mathcal{X}}
\def\Y{\mathcal{Y}}
\def\Z{\mathcal{Z}}
\def\P{\mathcal{P}}
\newcommand{\bone}{\mathbf 1}

\newcommand{\cD}{\mathcal D}

\newcommand{\cO}{\mathcal O}

\newcommand{\cX}{\mathcal X}
\newcommand{\cY}{\mathcal Y}

\newcommand{\E}{\mathbb{E}}

\newcommand{\dd}{\,\mathrm{d}}

\newcommand{\lrang}[1]{\left\langle {#1} \right\rangle}

\DeclareMathOperator\LSI{LSI}

\DeclareMathOperator\PI{PI}

\DeclareMathOperator\TV{TV}
\DeclareMathOperator\av{av}

\newcommand{\Var}{\operatorname{var}}

\DeclareMathOperator\sLSI{sLSI} % standard LSI appearing in SaloffCoste-Diaconis
\DeclareMathOperator\gPI{gPI}
\DeclareMathOperator\gLSI{gLSI}

\def\cg{\hat\mu} %coarse-grained dynamics
\def\eff{\eta} % effective dynamics
\def\aver{\mu^{\av}}
\def\stat{\rho} % invariant measure
\def\eps{\varepsilon} % epsilon
\def\sG{G} %slow-generator to replace C

\long\def\comm#1{{\color{orange}}}

\newcommand{\comment}[1]{} % remove details

\title{Non-equilibrium functional inequalities for finite Markov chains}
\author{Bastian Hilder\thanks{Department of Mathematics, Technische Universit\"at M\"unchen, Boltzmannstrasse 3, 85748, Garching b.\  M\"unchen, Germany.\\ Email: \href{mailto:bastian.hilder@tum.de}{bastian.hilder@tum.de}}, 
Patrick van Meurs\thanks{Faculty of Mathematics and Physics, Kanazawa University, Kakuma, Kanazawa 920-1192, Japan.\\ Email: \href{mailto:pjpvmeurs@staff.kanazawa-u.ac.jp}{pjpvmeurs@staff.kanazawa-u.ac.jp}}, 
Upanshu Sharma\thanks{School of Mathematics and Statistics, University of New South Wales, Sydney 2052, Australia.\\ Email: \href{mailto:upanshu.sharma@unsw.edu.au}{upanshu.sharma@unsw.edu.au}}}
\date{}

\begin{document}

\maketitle

\comm{These orange comments are NOT to be commented out! You can toggle them on/off in the preamble.}

\abstract{
Functional inequalities such as the Poincar\'e and log-Sobolev inequalities quantify convergence to equilibrium in continuous-time Markov chains by linking generator properties to variance and entropy decay. However, many applications, including multiscale and non-reversible dynamics, require analysing probability measures that are not at equilibrium, where the classical theory tied to steady states no longer applies.
We introduce generalised versions of these inequalities for arbitrary positive measures on a finite state space, retaining key structural properties of their classical counterparts. In particular, we prove continuity of the associated constants with respect to the reference measure and establish explicit positive lower bounds.
As an application, we derive quantitative coarse-graining error estimates for non-reversible Markov chains, both with and without explicit scale separation, and propose a quantitative criterion for assessing the quality of coarse-graining maps.
\newline
\newline
\textbf{Keywords.} functional inequalities; relative entropy; Fisher information; continuous-time Markov chains; coarse-graining
\newline
\textbf{Mathematics Subject Classification (2020).} 39B05; 39B62; 60J27; 60J28; 34C29; 34E13
} 

% 39B05: General theory of functional equations and inequalities
% 39B62: Functional inequalities, including subadditivity, convexity, etc.
% 60J45: Variational inequalities
% 60J27: Continuous-time Markov processes on discrete state spaces
% 60J28: Applications of continuous-time Markov processes on discrete state spaces
% 34C29: Averaging method for ordinary differential equations
% 34E13: Multiple scale methods for ordinary differential equations

\section{Introduction}

Functional inequalities such as the Poincar\'e inequality (PI) and the log-Sobolev inequality (LSI) are central tools in the quantitative analysis of Markov semigroups and diffusion processes. In the setting of a continuous-time Markov chain with an irreducible generator matrix, these inequalities quantify how fast observables and distributions relax toward the steady state (or towards equilibrium). The PI controls the exponential decay of the variance of observables, thereby determining the rate of convergence to steady state in the $L^2$-sense. The LSI provides a stronger, entropy-based control, implying exponential convergence of the law of the process to equilibrium in relative entropy. In both cases, the inequality constants (respectively called the PI and LSI constants) encapsulate the strength of the mixing mechanism in that small(/large) constants correspond to slow(/fast) 
equilibration. These inequalities have thus become fundamental in the analysis of convergence to equilibrium, concentration of measure, and the stability of stochastic systems, see~\cite{diaconisSaloffCoste96,Saloff-Coste97,AldousFill99,bobkovTetali06,MontenegroTetali06,erbarFathi18,schlichting2019} for a non-exhaustive list. 
Most of these references focus on the special setting of reversible Markov chains which 
allow for considerably deeper analysis.

However, the assumptions of reversibility and equilibrium (characterised by the steady state) is restrictive in certain applications.
For instance, complex chemical and biomolecular systems are often modelled via high-dimensional, possibly non-reversible, Markov chains constructed from molecular dynamics~\cite{Prinz11,SchutteSarich13}, which are routinely coarse-grained into clusters~\cite{DeuflhardWeber05,KubeWeber07,FackeldeySikorskiWeber18}. Inspired by coarse-graining approaches developed for  diffusion processes \cite{legollLelievre10,Chorin2003}, the recent work~\cite{HilderSharma24} by two of the authors provides a reduced dynamics on the clusters by defining an effective generator through conditional expectations. The accuracy of this (and related) reduced models depends on the mixing properties within each cluster, governed by measures that are \emph{not} the steady states. In these contexts, the classical PI and LSI inequalities, defined with the steady state as the reference measure, are no longer applicable and therefore lead to the following natural question.

\begin{center}
\emph{What is a meaningful generalisation of functional inequalities} \\
\emph{for reference measures other than the steady state? }
\end{center}

To the best of our knowledge, a comprehensive framework for non-equilibrium functional inequalities is not available in the literature. Therefore, in this work, we introduce generalised versions of the Poincar\'e and log-Sobolev inequalities (called gPI and gLSI) defined for arbitrary reference measures in the class of positive probability measures on the state space of the underlying continuous-time  Markov chain, see \eqref{def:genPI} and \eqref{def:genLSI}. These inequalities are built upon natural generalisations of the Dirichlet form and the Fisher information. We show that these latter generalisations retain key analytical properties of their classical counterparts: non-negativity, convexity, and continuity in their argument, see Proposition \ref{prop:GenDirFI-prop}. Using these generalised functionals, we establish a family of functional inequalities parametrised by the reference probability measure $\zeta$, whose optimal constants are denoted by $\alpha_{\gPI}(\zeta)$ and $\alpha_{\gLSI}(\zeta)$, see \eqref{def:PI-LSI-con}. Beyond the use of non-equilibrium reference measures, a crucial point of this study is that we do not require the Markov chain to be reversible (i.e.~to be in detailed balance), which is often assumed when dealing with such functional inequalities.  

We demonstrate that these generalised constants possess several desirable properties. As in the equilibrium case, there is a hierarchy for the functional inequalities, cf.~\cite{bobkovTetali06}. Specifically, the gLSI implies the gPI with $\alpha_{\gLSI}(\zeta) \leq \alpha_{\gPI}(\zeta)$, see Proposition \ref{prop:LSI-implies-Poincare}. Furthermore, both the gPI and gLSI inequalities satisfy the tensorisation property, see Proposition~\ref{prop:tensor}. Beyond these classical properties, we prove that the constants vary continuously with $\zeta$, see Proposition \ref{prop:gPI-char} and Theorem \ref{thm:continuity-alpha-LSI}. 
In addition, we provide strictly positive lower bounds for the constants, see Theorem~\ref{thm:pos}.
Specifically, these lower bounds also hold for the classical PI and LSI constants, by which we establish their positivity for irreducible generators that need not be reversible.
These properties require $\zeta$ to be strictly positive. However, this is typically guaranteed in applications where $\zeta$ is linked to solutions of forward Kolmogorov equations with irreducible generators, see for instance~\cite[Lemma~C.1]{HilderPeletierSharmaTse20}. 

We demonstrate the utility of these generalised inequalities by applying them to two problems where non-equilibrium analysis is essential. First, we derive quantitative stability estimates for the time-dependent probability distribution of a Markov chain relative to any other such measure $\zeta$, see Section \ref{sec:application-long-term-estimates}. Here $\zeta$ is time-dependent, and consequently the gPI and gLSI constants are time-dependent. We apply our continuity result and lower bounds to establish an exponential convergence result on the solutions to the Markov chains.

Second, we use the generalised functional inequalities to study clustering in Markov chains. In several works~\cite{legollLelievre10,ZhangHartmannSchutte16,LegollLelievreOlla17,DLPSS18,LegollLelievreSharma18,HartmannNeureitherSharma20,HilderSharma24} on coarse-graining estimates, including some of our own, functional inequalities are imposed for conditional measures that are not steady states of the underlying dynamics. Although this hypothesis has led to strong and elegant results, its structural meaning in a genuinely non-equilibrium setting has remained unclear. The present framework is a first step towards alleviating this issue by providing a systematic interpretation of such inequalities for continuous-time Markov chains. To demonstrate its applicability, we show that the LSI assumption with respect to non-steady states in~\cite{HilderSharma24} can be dropped entirely using our framework, where we specifically use the lower bounds. This provides quantitative coarse-graining error estimates both in the absence and presence of explicit scale separation in Markov chains, see Proposition~\ref{prop:error-estimate-no-scale-sep} and Theorem~\ref{thm:error-estimate-ms}.
Finally, our generalised inequalities allow us to provide a direct criterion for assessing the quality of a coarse-graining map, see Section \ref{sec:different-cg-maps}. In both of the aforementioned applications, the continuity of the functional inequality constants plays a central role in controlling deviations from the steady state.

Finally, we discuss how our functional inequalities can potentially be useful for constructing coarse-graining maps for Markov chains, and for $\Gamma$-calculus. In addition, we discuss how our framework can potentially be generalised to other settings, including countably infinite state spaces, and coarse-grained diffusion processes under appropriate additional assumptions. For the details on this, we refer to Section \ref{sec:discussion}.

In summary, this work extends the scope of classical functional inequalities to the non-stationary realm. It provides an analytical framework for studying -- at non-equilibrium -- relaxation, comparison, and reduction of continuous-time Markov chains.

\paragraph{Outline of the article and summary of notation.} In Section~\ref{sec:Not} we recall the classical notions of the PI and the LSI. Section~\ref{sec:Gen} presents the generalised functionals and introduces the generalised Poincar\'e and log-Sobolev constants along with their basic properties. Section \ref{sec:properties} contains our main results on important properties of the generalised constants. In Section~\ref{sec:applications} we demonstrate the utility of these results in two applications, and in Section \ref{sec:discussion} we speculate how these results could be applied to related research areas. Finally, Appendix \ref{app:failure-classical-inequalities} demonstrates that the naive generalisation of the classical inequalities generally fail to give a well-defined theory.

The following table summarises important notation used throughout this paper. Additional  notation will be introduced in Section~\ref{sec:notation}. 
\begin{center}
\begin{longtable}{@{\extracolsep{\stretch{1}}}*{3}{l}@{}}
\toprule
$\X$, $\Y$, $\Z$ & finite state space & \\
$|\Z|$ & cardinality of a finite space $\Z$ & \\
$\R^{|\Z|}_{>0}$ & vectors in $\R^{|\Z|}$ with strictly positive coordinates & \\
$\1$ & constant unit function & Sec.~\ref{sec:notation}\\
$\P(\Z)$ & space of probability measures on $\Z$ & \\
$\P_+(\Z)$ & space of strictly positive probability measures on $\Z$ & \\
$\eta_*$ & lower bound for $\eta\in \P_+(\Z)$ & Thm.~\ref{thm:pos} \\
$\cD_\eta(\Z)$ & space of probability densities with respect to $\eta\in \P_+(\Z)$  & \eqref{def:SpaceDen} \\
$M,L$ & irreducible generators for continuous-time Markov chains & \\
$M_*$ & smallest positive entry in generator $M$ & Thm.~\ref{thm:pos} \\
$M_\zeta$ & $\zeta$-symmetrised version of generator $M$ & \eqref{eq:symGen} \\
$\pi$ & steady state of continuous-time Markov chain & \\
$\Var_\eta$ & Variance with respect to probability measure $\eta$ & \eqref{def:Var} \\
$\RelEnt_\eta$ & relative entropy with respect to probability measure $\eta$ & \eqref{def:relEnt}\\
$\Ent_\eta$ & centred entropy with respect to probability measure $\eta$ & \eqref{eq:CentEnt}\\
$\calE$ & classical Dirichlet form with respect to the steady state $\pi$ & \eqref{def:Dirichlet}  \\
$\calE_\zeta$ & generalised Dirichlet form with respect to probability measure $\zeta$ & \eqref{def:genDirichlet} \\
$\RF$ & classical Fisher information with respect to the steady state $\pi$ & \eqref{def:Fisher}\\
$\RF_\zeta$ & generalised Fisher information with respect to probability measure $\zeta$ & \eqref{def:genFisher}\\
$\alpha_{\PI}$, $\alpha_{\gPI}$ & classical and generalised Poincar\'e constants &  \eqref{def:class-PI-con},~\eqref{def:PI-LSI-con}\\
$\alpha_{\LSI}$, $\alpha_{\gLSI}$ & classical and generalised log-Sobolev constants & \eqref{def:class-LSI-con},~\eqref{def:PI-LSI-con} \\
$\alpha_{\sLSI}$ & standard LSI constant & \eqref{eq:altLSI-intro} \\
$\xi$ & coarse-graining map & \eqref{eq:CGDyn-def} \\
$\Lambda_y$ & $y$-level set of $\xi : \X \to \Y$ & Sec.~\ref{sec:estimates-without-scale-sep} \\
\bottomrule                            
\end{longtable}
\end{center}

\section{The classical functional inequalities}\label{sec:Not}

In this section, we introduce the classical notions of the PI and the LSI. Here and elsewhere in the paper, we 
fix $\Z$ as a finite state space with $|\Z| \geq 2$ the number of states, and fix $M\in \R^{|\Z|\times|\Z|}$ as an irreducible, but not necessarily reversible, generator (or transition rate matrix) on $\Z$.
We will interchangeably interpret $M$ as a matrix or as an operator. In addition, let $\P(\Z)$ be the space of probability measures on $\Z$. Similar to $M$, we will treat measures on $\Z$ as either vectors in $\R^{|\Z|}$ or functions. The law $t\mapsto \mu_t\in \P(\Z)$ of the Markov chain driven by $M$ evolves according to the forward Kolmogorov equation
\begin{align}\label{eq:Nota-forKol}
\left\{\begin{aligned}
    \frac{\dd\mu}{\dd t} &= M^T \mu, \\
    \mu\big|_{t=0} &=\mu_0,
\end{aligned}\right.
\end{align}
with initial data $\mu_0\in\P(\Z)$. Irreducibility ensures that this Markov chain admits a unique strictly positive steady state $\pi\in\P_+(\Z)$, i.e.\ $M^T\pi =0$. Here, we denote the space of strictly positive probability measures by $\P_+(\Z)$.

To define the PI we introduce 
$$\mathbb E_\pi[f]\coloneqq \sum_{z\in \Z} f(z)\pi(z)$$
for the expectation of $f \in \R^{|\Z|}$ with respect to $\pi$ and 
$$(f,g)_\pi \coloneqq \sum_{z\in \Z} f(z)g(z)\pi(z)$$ 
for the the $\pi$-weighted inner product on $\R^{|\Z|}$.
\begin{defi}
For $f\in \R^{|\Z|}$ we define the \emph{variance} and \emph{Dirichlet form} of $f$ 
respectively  as
\begin{align}
    \Var_\pi(f)&\coloneqq  \mathbb E_\pi[f^2] - \bigl(\mathbb E_\pi[f]\bigr)^2 = (f^2,1)_\pi - \bigl((f,1)_\pi\bigr)^2, \label{def:Var}\\
     \calE_\pi(f,M) &\coloneqq (f,-Mf)_\pi 
        = - \sum_{z,z'\in\Z} f(z)f(z')M(z,z') \pi(z)
        = \frac12 \sum_{z,z'\in\Z} \big( f(z) - f(z') \big)^2 M(z,z') \pi(z). \label{def:Dirichlet}
\end{align} 
\end{defi}

Henceforth, for $\calE_\pi(f,M)$ and other objects depending on $M$, we drop $M$ from the notation and write $\calE_\pi(f)$ if there is no risk of confusion about the underlying generator. 
 
Classically, the Dirichlet form is defined as $\tilde \calE_\pi(f,g)=(f,-Mg)_\pi$ and requires $M$ to be reversible. Yet, throughout this paper we only deal with $f=g$, and therefore introduce the Dirichlet form as a function of one variable. Since we do not require $M$ to be reversible, $\calE_\pi$ is not a Dirichlet form in the classical sense. Nevertheless, it is typically still called a Dirichlet form by a slight abuse of notation. 

\begin{defi}
    The (classical) \emph{Poincar\'e inequality (PI)} is satisfied if there exists a constant $\alpha>0$ such that
\begin{equation}\label{def:cPI}
    \forall f\in \R^{|\Z|}: \ \ \Var_\pi(f)\leq \frac{1}{\alpha}\calE_\pi(f).
\end{equation}
The supremum over all possible choices of  $\alpha$ is called the \emph{Poincar\'e constant}. We denote it by
\begin{equation}\label{def:class-PI-con}
    \alpha_{\PI}(M) \coloneqq \inf_{\substack{f\in\R^{|\Z|} \\ f\notin \langle\1\rangle}} \frac{\calE_\pi(f,M)}{\Var_\pi(f)}.
\end{equation}
\end{defi}

Next, we introduce the classical notion of the log-Sobolev inequality. Throughout this article, let $\R^{|\Z|}_{>0}$ be the space of vectors in $\R^{|\Z|}$ with strictly positive coordinates. Then, let
\begin{align}\label{def:SpaceDen}
    \cD_\pi(\Z) \coloneqq \bigl\{\varphi \in \R_{>0}^{|\Z|} \, : \,  \mathbb E_\pi[\varphi] = 1\bigr\}
\end{align}
be the space of positive densities with respect to $\pi$. 
In the following, we use $f,g$ for functions on $\R^{|\Z|}$ and $\varphi,\psi$ for densities with respect to some probability measure on $\R^{|\Z|}$. 
\begin{defi}
     For $\varphi\in \cD_\pi(\Z)$ we define the \emph{relative entropy} and ($M$-)\emph{Fisher information} 
     respectively as
    \begin{align}
        \RelEnt_{\pi}(\varphi) &\coloneqq \sum_{z\in\Z}  \varphi(z)\log \varphi(z) \pi(z), \label{def:relEnt} \\
        \RF_\pi(\varphi,M) &\coloneqq \bigl(\varphi , -M \log \varphi\bigr)_{\pi} =  -\sum_{z,z'\in\Z} \varphi(z) \log \varphi(z') M(z,z') \pi(z). \label{def:Fisher}
    \end{align}
\end{defi}
\begin{defi}
    The \emph{log-Sobolev inequality (LSI)} is satisfied if there exists a constant $\alpha > 0$ such that 
\begin{equation}\label{def:cLSI}
\forall \varphi\in \cD_\pi(\Z): \ \ \RelEnt_\pi(\varphi)\leq  \frac{1}{\alpha} \RF_\pi(\varphi,M).
\end{equation}
The supremum over all possible choices of $\alpha$ is called the \emph{log-Sobolev constant}. We denote it by
\begin{equation}\label{def:class-LSI-con}
    \alpha_{\LSI}(M) \coloneqq \inf_{\substack{\varphi\in \cD_\pi(\Z)\\\varphi\neq \1}}\frac{\RF_\pi(\varphi,M)}{\RelEnt_\pi(\varphi)}.
\end{equation}
\end{defi}

Strictly speaking,~\eqref{def:cLSI} is the so-called modified LSI, as opposed to the standard LSI introduced in Remark~\ref{rem:LitLSI-PI}. Since this paper predominantly deals with generalising~\eqref{def:cLSI}, we simply refer to it as the LSI instead of the modified LSI. We recall that, similar to the abbreviated notation $\calE_\pi(f)$, we simply write $\alpha_{\PI},\,\RF_\pi(\varphi),\,\alpha_{\LSI}$ whenever the underlying generator is clear from the context.

A key difference between the two functional inequalities above is the domains of functions $f$ and $\varphi$ over which these inequalities are defined. The following remark discusses this. 
\begin{rem}\label{rem:cenEnt}
    First, the relative entropy $\RelEnt_\pi(f)$ is only defined for $f\in \R^{|\Z|}_{\geq 0}$ and the Fisher information $\RF_\pi(f)$ only for $f\in \R^{|\Z|}_{>0}$. Thus, the full class of $f \in \R^{|\Z|}$ in the PI cannot be considered for the LSI.

    Second, for the LSI, the inequality in \eqref{def:cLSI} could be considered for all $f\in \R^{|\Z|}_{>0}$ rather than the restricted class $\varphi\in \cD_\pi(\Z) \subset \R^{|\Z|}_{>0}$.
    Yet, this inequality would be meaningless to consider as the LSI. To see this, note that any
    $f\in \R^{|\Z|}_{>0}$ can be uniquely decomposed as $ \beta \varphi$ with $ \beta > 0$ and $\varphi\in \cD_\pi(\Z)$; in particular $ \beta = \mathbb E_{\pi}[f]$ and $\varphi = \frac{f}{\mathbb E_{\pi}[f]}$. Then
    \begin{align*}
        \RelEnt_\pi(f) =  \beta \RelEnt_\pi(\varphi) + \beta\log\beta, \ \ \RF_\pi(f) = \beta\RF_\pi(\varphi).
    \end{align*}
Consequently, the LSI constant would read
    \begin{equation*}
        \inf_{\varphi \in  \cD_\pi(\Z)} \inf_{\beta > 0} \frac{\RF_\pi(\varphi)}{\RelEnt_\pi(\varphi)+\log\beta},  
    \end{equation*}
    which equals $- \infty$ due to the minimisation over $\beta$.

    This issue can be fixed by working with the centred entropy instead, defined as 
    \begin{equation}\label{eq:CentEnt}
        \Ent_\pi(f) \coloneqq \RelEnt_\pi(f) - \mathbb E_\pi(f) \log\bigl(\mathbb E_\pi(f)\bigr). 
    \end{equation}
    In this case $\Ent_\pi(\beta\varphi)=\beta\RelEnt_\pi(\varphi)$ and thus
    \[
      \alpha_{\LSI} 
      = \inf_{\varphi \in  \cD_\pi(\Z)} \frac{\RF_\pi(\varphi)}{\RelEnt_\pi(\varphi)}
      = \inf_{f\in \R^{|\Z|}_{>0}} \frac{\RF_\pi(f)}{\Ent_\pi(f)}.
    \]
    Consequently, all results pertaining to  the LSI inequality can be extended to $\R^{|\Z|}_{>0}$ using this centred entropy. However, since we only consider the case of densities, where $\Ent_\pi(\varphi) = \RelEnt_\pi(\varphi)$ for $\varphi\in D_\pi(\Z)$, we use the setting \eqref{def:cLSI} throughout this paper.
    \end{rem}

The following remark provides some context for these functional inequalities in the setting of Markov chains. 

\begin{rem}\label{rem:LitLSI-PI}
The Poincar\'e and log-Sobolev inequality have been used to quantify the convergence to equilibrium in Markov chains, see~\cite{bobkovTetali06} and references therein. In particular it follows that the solution $\mu_t$ to the forward Kolmogorov equation \eqref{eq:Nota-forKol}  converges to the steady state $\pi$ exponentially fast in variance and relative entropy with rates $\alpha_{\PI}$ and $\alpha_{\LSI}$ respectively~\cite[Section~1]{bobkovTetali06}, due to the relations 
\begin{equation} \label{ddt:var:H}
  \begin{aligned}
      \frac{\dd}{\dd t}\Var_\pi \Big(\frac{\mu_t}{\pi} \Big)
      &= -2 \calE_\pi\Big(\frac{\mu_t}{\pi}\Big) 
  \leq - 2 \alpha_{\PI} \Var_\pi\Big(\frac{\mu_t}{\pi}\Big), \\
  %\qquad \text{and} \qquad 
  \frac{\dd}{\dd t}\RelEnt_\pi \Big(\frac{\mu_t}{\pi} \Big)
  &= -\RF_\pi\Big(\frac{\mu_t}{\pi}\Big)
  \leq - \alpha_{\LSI} \RelEnt_\pi\Big(\frac{\mu_t}{\pi}\Big). 
  \end{aligned}  
\end{equation}
Regarding both constants, $0<\alpha_{\LSI}\leq 2 \alpha_{\PI}$ (see the last line of the proof of \cite[Proposition~3.6]{bobkovTetali06}). Furthermore, in the special case of reversible Markov chains, $\alpha_{\PI}$ is the spectral-gap associated to the generator matrix $M$~\cite[Section~2.2]{diaconisSaloffCoste96}. 

We also point out that the Fisher information is connected to the Dirichlet form via
\begin{equation*}
    \RF_\pi(\varphi) = \bigl(\varphi,-M\log\varphi \bigr)_\pi \geq \frac12 \calE_\pi(\sqrt{\varphi}).
\end{equation*}
This gives rise to the \textit{standard LSI (sLSI)} (we follow the naming convention introduced in~\cite{bobkovTetali06}), which measures the ratio between $\calE_\pi(f)$ and $\Ent_\pi(f^2)$ (introduced in~\eqref{eq:CentEnt}). The corresponding constant is given by~\cite[Equation~(1.7)]{diaconisSaloffCoste96}
\begin{equation} \label{eq:altLSI-intro}
    \alpha_{\sLSI}(M) \coloneqq \inf_{\substack{f\in\R^{|\Z|_{>0}}\\f\notin \langle\1\rangle}}\frac{\calE_\pi(f,M)}{\Ent_\pi(f^2)}.
\end{equation}
\end{rem}

\section{Generalised functional inequalities}\label{sec:Gen}

As stated in the introduction, our aim is to generalise the PI and the LSI to the non-equilibrium setting. This translates to defining these inequalities with the steady state $\pi$ replaced by a general strictly positive probability measure $\zeta\in\P_+(\Z)$. 

We first point out that the naive generalisation of replacing in the PI and the LSI $\pi$ by $\zeta$ does not work, even though the definitions of all functionals can directly be extended to any $\zeta \in\P_+(\Z)$. We focus on the classical Poincar\'e inequality~\eqref{def:cPI}, whose naive generalisation reads as 
\begin{equation}\label{def:class-PI-con-genNaive}
    \tilde \alpha_{\PI}(\zeta, M) \coloneqq \inf_{\substack{f\in\R^{|\Z|} \\ f\notin \langle\1\rangle}} \frac{\calE_\zeta(f, M)}{\Var_\zeta(f)},
\end{equation}
where $\langle\1\rangle$ denotes the space of constant functions. As for the classical inequality, we
need to impose $f\notin \langle\1\rangle$ to ensure that $\Var_\zeta(f)\neq 0$. A simple calculation (see Proposition~\ref{prop:cFI-Fail} in the appendix) then shows that $\tilde \alpha_{\PI}(\zeta) = -\infty$ for any $\zeta\neq \pi$.

In this section we first introduce generalisations of the Dirichlet form and the Fisher information, which are then used to define non-equilibrium versions of Poincar\'e and log-Sobolev inequalities. We conclude this section by proving a number of properties of these generalisations.  

\subsection{Notation}\label{sec:notation} 

Other than the Dirichlet form and the Fisher information, we extend the other functionals and notation from Section \ref{sec:Not} by replacing $\pi$ in their definitions by $\zeta \in\P_+(\Z)$. In this manner we define $\mathbb E_\zeta$, $(f,g)_\zeta$, $\Var_\zeta$, $\cD_\zeta(\Z)$ and $\RelEnt_{\zeta}$. In addition, we define for $f\in \R^{|\Z|}$
\begin{align*}
    \|f\|^2\coloneqq \sum_{z\in\Z} f^2(z), \quad \text{and} \quad \|f\|^2_\zeta =\sum_{z\in\Z} f^2(z)\zeta(z).
\end{align*}
Let $\1 : \Z \to \R$ be the constant function equal to $1$. For $f_1, \ldots, f_k \in \R^{|\Z|}$, we denote their span by $\langle f_1, \ldots, f_k \rangle$. For $\zeta \in \P_+(\Z)$ and $U \subset \R^{|\Z|}$ a subspace, we denote by $U^{\perp_\zeta}$ the orthogonal complement with respect to the inner product $(f,g)_\zeta$, e.g.
\begin{align*}
\langle\1\rangle^{\perp_\zeta} = \{f\in \R^{|\Z|}: \ (f,\1)_\zeta = 0 \}.  
\end{align*}
In a similar manner, we also introduce the orthogonal complement in the flat geometry as 
\begin{align*}
\langle\1\rangle^{\perp} = \Bigl\{f\in \R^{|\Z|}: \ (f,\1)=\sum_{z\in\Z} f(z) = 0 \Bigr\}.  
\end{align*}
In particular, since $\R^{|\Z|}= \langle\1\rangle \oplus \langle\1\rangle^{\perp_\zeta}$, we have 
\begin{align}\label{eq:Rz-directSum}
  \forall f\in \R^{|\Z|}, \ \exists! c\in \R, \ \exists! h\in \langle\1\rangle^{\perp_\zeta}: \ \   f=c\1 + h.
\end{align}
Finally, we set
   \[
     \zeta_* \coloneqq \min_{z \in \Z} \zeta(z) > 0.
   \]

\subsection{Generalised Dirichlet form and the Fisher information}\label{sec:def-gDF-gFI}

As already discussed, a key interpretation for the classical notions of Dirichlet form and the Fisher information is that these objects can be respectively derived as the dissipation of variance and relative entropy with respect to the steady state $\pi$ along the solution $t\mapsto \mu_t$ of the forward Kolmogorov equation~\eqref{eq:Nota-forKol}, see \eqref{ddt:var:H}.

Following this connection, to motivate generalisation towards non-equilibrium measures, we consider two time-dependent solutions $t\mapsto \mu_t,\zeta_t$ of~\eqref{eq:Nota-forKol}, and consider the \emph{non-equilibrium} evolution of variance and relative entropy, i.e.\ $\frac{\dd}{\dd t}\Var_{\zeta_t}(\frac{\mu_t}{\zeta_t})$ and $\frac{\dd}{\dd t}\RelEnt_{\zeta_t}(\frac{\mu_t}{\zeta_t})$. We want our generalisation to be such that these derivatives satisfy similar equations as in \eqref{ddt:var:H}, see \eqref{ddt:var:H:zeta} below. This leads to the following generalisations of the Dirichlet form and the Fisher information.  

\begin{defi}
Let $\zeta\in \P_+(\Z)$. For any $f\in \R^{|\Z|}$ we define the \emph{generalised Dirichlet form} with respect to $\zeta$  as 
\begin{equation}\label{def:genDirichlet}
    \calE_\zeta(f,M) 
    \coloneqq  \mfrac12 \sum_{z,z'\in\Z} M(z,z')\zeta(z)\bigl(f(z)-f(z')\bigr)^2 
    =  (f,-Mf)_\zeta + \mfrac12(1,Mf^2)_\zeta.
\end{equation}
For any $\varphi\in \cD_\zeta(\Z)$ (defined in~\eqref{def:SpaceDen}) we define the \emph{generalised Fisher information} with respect to $\zeta$ as 
\begin{equation}\label{def:genFisher}
    \RF_\zeta(\varphi,M) \coloneqq 
    \sum_{z,z'\in \Z} M(z,z')\zeta(z) \varphi(z) \biggl[ \frac{\varphi(z')}{\varphi(z)} - 1 -\log\biggl(\frac{\varphi(z')}{\varphi(z)}\biggr)\biggr]. 
\end{equation}
\end{defi}
As in Section \ref{sec:Not}, we write $\calE_\zeta(f),\,\RF_\pi(\varphi)$ in place of $\calE_\zeta(f,M),\,\RF_\pi(\varphi,M)$ if there is no risk of confusion about the underlying generator.

Following our motivation above, it is easy to check that for two solutions $t\mapsto \mu_t,\zeta_t$ of~\eqref{eq:Nota-forKol}, we indeed have the property 
\begin{equation} \label{ddt:var:H:zeta}
 \frac{\dd}{\dd t}\Var_{\zeta_t}\biggl(\frac{\mu_t}{\zeta_t}\biggr)= -2\calE_{\zeta_t}\biggl(\frac{\mu_t}{\zeta_t}\biggr), \ \ \frac{\dd}{\dd t}\RelEnt_{\zeta_t}\biggl(\frac{\mu_t}{\zeta_t}\biggr)= -\RF_{\zeta_t}\biggl(\frac{\mu_t}{\zeta_t}\biggr),   
\end{equation}
i.e. the generalised Dirichlet form and the generalised Fisher information are indeed the dissipation of \emph{non-equilibrium} variance and relative entropy, i.e.\ dissipation along non steady-state solutions of~\eqref{eq:Nota-forKol}. Furthermore, these objects reduce to their classical counterparts~\eqref{def:Var} and \eqref{def:Fisher} when $\zeta$ is replaced by the steady state $\pi$, i.e.\ when $M^T\zeta=0$. 
Regarding the generalised Dirichlet form, it equals the right-hand side of \eqref{def:Dirichlet}. Yet, the last equality in \eqref{def:Dirichlet} relies on $\pi$ being invariant. The second equality in \eqref{def:genDirichlet} shows the additional term which is needed to have equality for non steady states.
Regarding the generalised Fisher information, we note that it can be derived as a \emph{modulation} of its classical counterpart~\cite[Section~5.1]{hilder17}. 
These factors explain the label of `generalisations' that has been used above. 

We further remark that $\calE_\zeta$, similar to $\calE_\pi$, is not a Dirichlet form in the strict sense, and that
the domain of definition for $\calE_\zeta$ and $\RF_\zeta$ with respect to $\zeta$ can be further enlarged to general probability measures by extending them via lower-semicontinuous envelopes. We skip these details here for simplicity of presentation. 

The generalised Dirichlet form and the Fisher information share several properties with their classical variants. We summarise these in the following result. 

\begin{prop}\label{prop:GenDirFI-prop}
    For any $\zeta\in \P_{+}(\Z)$, the generalised Dirichlet form~\eqref{def:genDirichlet} and the generalised Fisher information~\eqref{def:genFisher} with respect to $\zeta$  satisfy the following.
    \begin{enumerate}
        \item Continuity: The maps $(\zeta,f)\mapsto \calE_\zeta(f)$ and $(\zeta,\varphi)\mapsto \RF_\zeta(\varphi)$ are continuous on $\P_+(\Z)\times \R^{|\Z|}$ and $\{(\zeta,\varphi) \,:\,\zeta\in \P_+(\Z), \varphi\in  \cD_\zeta(\Z)\}$ respectively. 
        % \US{Is this domain definition ok?} \pvm{Yes}
        \item Non-negativity: $\calE_\zeta(f),\RF_\zeta(\varphi)\geq 0$ with equality if and only if $f \in \lrang{\bone}$ and $\varphi=\1$ respectively.
        \item Convexity: $\calE_\zeta(\cdot)$ and $\RF_\zeta(\cdot)$ are convex on their domains of definition.
        \item Connection: Let $f\in\R^{|\Z|}$ with $\E_\zeta[f]=0$. Then, for all $\delta \in \R$ with $|\delta| \leq \frac12 \|f\|_\infty^{-1}$, we have $\1+ \delta f \in \cD_\zeta(\Z)$ and
        \begin{equation}\label{eq:genRF-lin}
            \RF_\zeta(\1+ \delta f) = \delta^2 \bigl(\calE_\zeta(f) + \delta R \bigr),
        \end{equation}
        where the remainder term satisfies $|R| \leq 3 \|M\|_\infty |\Z| \|f\|_\infty^3$.  
    \end{enumerate}
\end{prop}

\begin{proof}
    The continuity of the generalised Dirichlet form and the generalised Fisher information follows directly from their respective definitions, see \eqref{def:genDirichlet} and \eqref{def:genFisher}, respectively.

    Next we show the non-negativity and convexity of $\calE_{\zeta}$. The non-negativity of $\calE_\zeta(\cdot)$ follows from the definition~\eqref{def:genDirichlet} since $M(z,z')\geq 0$ for $z\neq z'$ and the term $M(z,z)$ does not contribute to the summation. Additionally, we clearly have $\calE_\zeta(c\1) = 0$. On the other hand, if $\calE_{\zeta}(f) = 0$, then $M(z,z')(f(z) - f(z'))^2 = 0$ for every $z \neq z'$. Since $M$ is irreducible, there is a path $({z}_j)_{j=1}^J \subset \Z$ between every $z$ and $z'$ such that $M({z}_j,{z}_{j+1}) > 0$. This yields that $f$ is constant along this path. Since $z,z'$ were arbitrary, we conclude that $f$ is constant on $\Z$. Finally, the convexity of $\calE_\zeta$ follows directly from the convexity of $f \mapsto (f(z) - f(z'))^2$, which can be written as a composition of $x \mapsto x^2$ with a linear map.

    Next, we prove the non-negativity and convexity for $\RF_{\zeta}$. Let $\Phi(x) \coloneqq x - 1 - \log x$ for $x > 0$. Note that $\Phi \geq 0$ is strictly convex and that $\Phi(1) = 0$. Then, the non-negativity of $\RF_{\zeta}$  follows from $M(z,z')\geq 0$. Similarly to $\calE_\zeta$, using that $\zeta, \varphi > 0$ and that $\Phi(x) = 0$ only at $x = 1$, we obtain that $\RF_{\zeta}$ vanishes if and only if $\tfrac{\varphi(z')}{\varphi(z)} = 1$ for all $z \neq z'$, which together with $\varphi \in \cD_\zeta$ shows $\varphi = \1$, see also \cite[Lemma~2.5]{HilderPeletierSharmaTse20}. To show convexity, we observe that it is sufficient to show convexity of $\varphi \mapsto \varphi(z) \Phi(\tfrac{\varphi(z')}{\varphi(z)})$. This either follows from direct calculation using that $(x,y) \mapsto - x \log(y/x)$ is convex on $\R^2_{>0}$, see also \cite[Remark~2.12]{HilderPeletierSharmaTse20}, or using an abstract result on perspective functions, see e.g.~\cite{Combettes18}.
   
    Now we prove the final part. Since $\| \delta f \|_\infty \leq \frac12 < 1$ and $\mathbb E_\zeta[f] = 0$, we have $\varphi \coloneqq \1+\delta f \in\cD_\zeta(\Z)$, using which 
    \begin{multline*}
        \RF_\zeta(\1+\delta f) = \sum_{\substack{ z,z'\in\Z \\ z \neq z'}} M(z,z')\zeta(z) \bigl[ \bigl(1+\delta f(z')\bigr) - \bigl(1+\delta f(z)\bigr) \\
        - \bigl(1+\delta f(z)\bigr) \log\bigl(1+\delta f(z')\bigr)
        + \bigl(1+\delta f(z)\bigr)\log\bigl(1+\delta f(z)\bigr)\bigr]. 
    \end{multline*}
    Taylor expanding $\log(1+x) = x - \frac12 x^2 - \sum_{k=3}^\infty \frac1k (-x)^k$ leads to  
    \begin{align*}
        \RF_\zeta(\1+\delta f) &=  \frac{\delta^2}{2} \sum_{\substack{ z,z'\in\Z \\ z \neq z'}} M(z,z') \zeta(z) \bigl(f(z)-f(z')\bigr)^2+\mathcal \delta^3 R 
        = \delta^2 \bigl(\calE_\zeta (f) + \delta R \bigr),
    \end{align*} 
    where we estimate the remainder term 
\begin{multline*}
    R 
    \coloneqq \sum_{z,z'\in\Z} M(z,z') \frac{\zeta(z)}{\delta^3} \biggl[ 
        \frac{\delta^3}2 f(z) \bigl( f(z')^2 - f(z)^2 \bigr) 
        + \bigl(1+\delta f(z)\bigr) 
        \sum_{k=3}^\infty \frac{(-\delta)^k f(z')^k}k \\
        - \bigl(1+\delta f(z)\bigr) \sum_{k=3}^\infty \frac{(-\delta)^k f(z)^k}k \biggr]
\end{multline*}
as (with $F \coloneqq \|f\|_\infty$ and $|\delta| F \leq \frac12$)
\begin{align*}
    |R|
    &\leq \bigg| \sum_{\substack{ z,z'\in\Z \\ z \neq z'}} M(z,z') \zeta(z) \bigg| \biggl(F^3 + \frac2{|\delta|^3} (1 + |\delta| F) \sum_{k=3}^\infty \frac{(|\delta| F)^k}k \biggr) \\
    &\leq \|M\|_\infty |\Z| F^3 \biggl(1 + 3 \sum_{\ell=0}^\infty \frac{(|\delta| F)^\ell}{\ell + 3} \biggr)
    \leq \|M\|_\infty |\Z| F^3 \biggl(1 + \sum_{k=0}^\infty \frac1{2^k} \biggr).
\end{align*}
This shows that, indeed, $R$ is uniformly bounded in $\delta$ and that the asserted estimate holds.
\end{proof}

Regarding \eqref{eq:genRF-lin}, a similar relation holds between $\RelEnt_\zeta$ and $\Var_\zeta$. Indeed, under the same assumptions as for \eqref{eq:genRF-lin}, we have
 \begin{equation} \label{eq:genRE-lin}
     \RelEnt_\zeta(\1+\delta f) = \delta^2 \Big( \frac{1}{2}\Var_\zeta(f) + \delta \tilde{R} \Big), 
 \end{equation}
where the remainder term $\tilde{R}$ satisfies $|\tilde{R}| \leq \frac13 \|f\|_\infty^3$. To see this, a simplified version of the proof of \eqref{eq:genRF-lin} reveals that
\begin{align*}
    \RelEnt_\zeta(\1+\delta f)
    = \sum_{z\in\Z} (\1+\delta f(z)) \bigg( \delta f(z) - \frac12 (\delta f(z))^2 - \sum_{k=3}^\infty \frac{(-\delta f(z))^k}k \bigg) \zeta(z)
    = \delta \mathbb E_\zeta[f] + \frac12 \delta^2 E_\zeta[f^2] + \delta^3 \tilde{R},
\end{align*}
where now
\begin{align*}
  \tilde{R} &= \sum_{z\in\Z}  \frac{\zeta(z)}{\delta^3} \biggl[ 
        - \frac{\delta^3 f(z)^3}2
        - \bigl(1+\delta f(z)\bigr) 
        \sum_{k=3}^\infty \frac{(-\delta)^k f(z)^k}k \biggr] \\
   &= \sum_{z\in\Z}  \frac{\zeta(z)}{\delta^3} \biggl[ 
        - \sum_{k=3}^\infty \frac{(-\delta f(z))^k}k
        + \sum_{k=2}^\infty \frac{(-\delta f(z))^{k+1}}k \biggr]
   = \sum_{z\in\Z}  \frac{\zeta(z)}{\delta^3} \sum_{k=3}^\infty \frac{(-\delta f(z))^k}{k(k-1)}.
\end{align*}
Then, applying $|f(z)| \leq \|f\|_\infty$, noting that $\sum_{z\in\Z} \zeta(z) = 1$, using $|\delta| \|f\|_\infty \leq \frac12$, and $k(k-1)\geq 6$ for $k\geq 3$ the asserted estimate on $|\tilde{R}|$ follows.

The following remark discusses the connection between the generalised Fisher information introduced above and yet another generalisation of the Fisher information introduced by some of the authors in~\cite{HilderPeletierSharmaTse20}. 
\begin{rem}\label{rem:alternative-generalistion}
    In~\cite[Definition~1.5]{HilderPeletierSharmaTse20}, the authors introduce a different generalisation of the Fisher information inspired by large deviations of independent copies of Markov chains. In the language of this article, for any $\lambda\in (0,1)$ and $f\in\R^{|\Z|}_{>0}$, this version of the Fisher information reads (note the difference in domain of definition compared to~\eqref{def:genFisher})
    \begin{equation*}
        \RF^\lambda_\zeta(f) \coloneqq \sum_{z,z'\in\Z} M(z,z') \zeta(z) f(z) \biggl[ \frac{f(z')}{f(z)} -1 - \frac{1}{\lambda}\biggl(\bigg( \frac{f(z')}{f(z)} \biggr)^\lambda- 1\biggr) \biggr]. 
    \end{equation*}
    Noting that
    \begin{equation*}
        \lim_{\lambda\to 0} \bigl[ x - 1 -\frac{1}{\lambda}(x^\lambda-1)\bigr] = x -1 - \lim_{\lambda\to 0} \biggl(\frac{e^{\lambda\log x}-1}{\lambda}\biggr) = x -1 -\log x
    \end{equation*}
    we obtain $\RF^\lambda_\zeta(f) \to \RF_\zeta(f)$ as $\lambda\to 0$, where the limit is defined in~\eqref{def:genFisher}. In particular, the results in this paper can be appropriately generalised if $\RF_\zeta$ is replaced by $\RF^\lambda_\zeta$ using that $x \mapsto x - 1 - \tfrac{1}{\lambda}(x^\lambda - 1)$ is convex and
    \begin{equation*}
        \RF^\lambda_\zeta(1 + \delta f) = \delta^2 ((1-\lambda)\calE_\zeta(f) + \mathcal{O}(\delta)),
    \end{equation*}
    which follows from a Taylor expansion of $x \mapsto x - 1 - \tfrac{1}{\lambda}(x^\lambda - 1)$.
\end{rem}

\subsection{Generalised PI and LSI inequalities}
With the generalised objects defined above, we are now ready to introduce the corresponding \emph{(non-equilibrium)} generalised functional inequalities. 
\begin{defi}
 Let $M\in \R^{|\Z|\times |\Z|}$ be an irreducible generator. A strictly positive probability measure $\zeta\in \P_+(\Z)$ satisfies the
\begin{itemize}
\item \emph{generalised Poincar\'e inequality (gPI)} if there exists a constant $\alpha > 0$ such that
\begin{equation}\label{def:genPI}
\forall f\in \R^{|\Z|}  : \ \ \Var_\zeta(f)\leq \frac{1}{\alpha} \calE_\zeta(f);
\end{equation}
\item \emph{generalised log-Sobolev inequality (gLSI)} if there exists a constant $\alpha > 0$ such that 
\begin{equation}\label{def:genLSI}
\forall \varphi\in \cD_\zeta(\Z): \ \ \RelEnt_\zeta(\varphi)\leq  \frac{1}{\alpha} \RF_\zeta(\varphi).
\end{equation}
where $\cD_\zeta(\Z)$ is the space of densities with respect to $\zeta$ (defined in~\eqref{def:SpaceDen}).
\end{itemize}
The suprema over all possible choices of $\alpha$ are called respectively the \emph{generalised Poincar\'e} and the \emph{generalised log-Sobolev constants}. We denote them by $\alpha_{\gPI}(\cdot, M), \alpha_{\gLSI}(\cdot, M) :\P_+(\Z)\to [0,\infty)$ and they are given by
\begin{align}\label{def:PI-LSI-con}
    \alpha_{\gPI}(\zeta, M) \coloneqq \inf_{\substack{ f\in\R^{|\Z|} \\ f\notin \langle\1\rangle }} \frac{\calE_\zeta(f, M)}{\Var_\zeta(f)}, \quad\text{and}\quad \alpha_{\gLSI}(\zeta, M) \coloneqq \inf_{\substack{\varphi\in\cD_\zeta(\Z) \\ \varphi\neq \1}} \frac{\RF_\zeta(\varphi, M)}{\RelEnt_\zeta(\varphi)}.
\end{align}
\end{defi}

As in the case of the classical variants, we write $\alpha_{\gPI}(\zeta),\alpha_{\gLSI}(\zeta)$ instead of $\alpha_{\gPI}(\zeta,M),\alpha_{\gLSI}(\zeta,M)$ if there is no risk of confusion about the underlying generator. Since we are particularly interested in the dependence on $\zeta$, we keep it in the notation.

Note that \eqref{def:genPI} and \eqref{def:genLSI} hold trivially for all $\zeta \in \P_+(\Z)$, all $f = c\1  \in \langle\1\rangle$ and $\varphi=\1$ respectively, since $\Var_\zeta(c\1)=\calE_\zeta(c\1)=0$ and $\RelEnt_\zeta(\1)=\RF_\zeta(\1)=0$. Hence, removing them in \eqref{def:PI-LSI-con} does not change the fact that $\alpha_{\gPI}(\zeta)$ and $\alpha_{\gLSI}(\zeta)$ are respectively the suprema over all constants $\alpha$ in \eqref{def:genPI} and \eqref{def:genLSI}.

We now prove the continuity of $\zeta \mapsto \alpha_{\gPI}(\zeta)$ as well as equivalent characterisations of the generalised PI constant (see Proposition \ref{prop:gPI-char} below), which will be useful in the forthcoming analysis. In the first two characterisations, the space of test functions is reduced. The first characterisation simplifies the fraction, and the second characterisation removes the dependence on $\zeta$ from the class over which $f$ is minimised. From these characterisations, it is relatively easy to prove continuity of $\zeta \mapsto \alpha_{\gPI}(\zeta)$.
The final, third characterisation draws a bridge to the formula of the generalised LSI constant thereby establishing a useful connection between the gPI and the gLSI.
The proof of this exploits the continuity of $\alpha_{\gPI}$ and that the generalised Dirichlet form and variance appear as the leading-order term of respectively the generalised Fisher information and the relative entropy close to $\varphi = \1$, see \eqref{eq:genRF-lin} and \eqref{eq:genRE-lin}. The result makes use of notation introduced in Section~\ref{sec:notation}.

\begin{prop}\label{prop:gPI-char}
    The gPI constant $\alpha_{\gPI}$ defined in~\eqref{def:PI-LSI-con} is continuous on $\P_+(\Z)$. Moreover, we have the following alternative characterisations for any $\zeta \in \P_+(\Z)$
    \begin{align}\label{eq:gPI-char}
        \alpha_{\gPI}(\zeta) 
        = \inf_{\substack{f\in \langle \1 \rangle^{\perp_\zeta} \\ \|f\|_\zeta=1}}\calE_\zeta(f)
        = \inf_{\substack{f\in \langle \1 \rangle^\perp \\ \|f\|=1}} 
        \frac{\calE_\zeta(f)}{\Var_\zeta(f)}  
        .
     \end{align}
    In addition, for any $\P_+(\Z) \ni \zeta_\e \to \zeta$ and any $0 < \delta_\e \to 0$ as $\e \to 0$, we have
    \begin{align}\label{eq:gPI-char:eps}
        \alpha_{\gPI}(\zeta) 
        = 
        \frac12 \lim_{\e \to 0} \inf_{\substack{\varphi \in \cD_{\zeta_\e}(\Z) \\ 0 < \| \varphi - \1 \|_{\zeta_\e} < \delta_\e }}
\frac{\RF_{\zeta_\e}(\varphi)}{\RelEnt_{\zeta_\e}(\varphi)}.
     \end{align}
\end{prop}

\begin{proof}
    Since $\R^{|\Z|}=\langle\1\rangle\oplus\langle\1\rangle^{\perp_\zeta}$, any $f\notin \langle\1\rangle$ can be decomposed into $f=c\1+g$  where $c\in\R$ and $g\in \langle\1\rangle^{\perp_\zeta}\backslash\{\mathbf{0}\}$; in particular $\E_\zeta[g]=0$. From this we claim that
    \begin{align*} %\label{eq:equiv-char-1}
        \alpha_{\gPI}(\zeta) = \inf_{g\in \langle\1\rangle^{\perp_\zeta}\backslash\{\mathbf{0}\}} \frac{\calE_\zeta(g)}{\Var_\zeta(g)} = \inf_{g\in \langle\1\rangle^{\perp_\zeta}\backslash\{\mathbf{0}\}} \calE_\zeta\biggl(\frac{g}{\|g\|_\zeta}\biggr)=\inf_{\substack{f\in \langle \1 \rangle^{\perp_\zeta} \\ \|f\|_\zeta=1}}\calE_\zeta(f).
    \end{align*}
    Indeed, the first equality follows by using $f=c\1+g$ as above, $\calE_\zeta(c\1+g)=\calE_\zeta(g)$ and $\Var_\zeta(c\1+g)=\Var_\zeta(g)$.  The second equality follows since $\calE_\zeta( \beta g)= \beta^2\calE_\zeta(g)$ for any $ \beta \in \R$ and $\Var_\zeta(g)=\E_\zeta[g^2]=\|g\|^2_\zeta$.   

    The second equality in~\eqref{eq:gPI-char} follows by repeating the same arguments as above using the alternative decomposition $\R^{|\Z|}=\langle\1\rangle\oplus\langle\1\rangle^{\perp}$ and noting that $\Var_\zeta( \beta g) = \beta^2 \Var_\zeta(g)$. 

    We now prove the continuity of $\alpha_{\gPI}$. Using the second alternative characterisation we write 
    \begin{align*}
        \alpha_{\gPI}(\zeta) = \inf_{\substack{f\in\langle\1\rangle^{\perp}\\\|f\|=1}} F(\zeta,f),
        \qquad
        F(\zeta,f)\coloneqq \frac{\calE_\zeta(f)}{\Var_\zeta(f)},
    \end{align*}
    where the domain of $F(\cdot,\cdot)$ is given by
    \begin{align*}
        \mathrm{Dom}(F) = \cup_{\tau>0} K_\tau, \ \ K_\tau \coloneqq \bigl\{\zeta\in \P_+(\Z): \zeta(z)\geq \tau, \ \forall z\in \Z\bigr\}  \times \bigl\{ f\in\langle\1\rangle^{\perp}:\|f\|=1 \bigr\}.
    \end{align*}
    Note that $F\in C(K_\tau)$ for any $\tau>0$ since $K_\tau$ is a compact set,  $\calE_\zeta(f)$ and $\Var_\zeta(f)$ are continuous on $K_\tau$ (the former follows from Proposition~\ref{prop:GenDirFI-prop}), and 
    for any $f\in \langle\1\rangle^{\perp}$ 
     \begin{align*}
        \Var_\zeta(f)
        = \mathbb E_\zeta \big[ ( f - \E_\zeta [f] )^2 \big]
        \geq \tau \| f - \E_\zeta [f] \|^2
        = \tau \big( \| f \|^2 + \| \E_\zeta [f] \|^2 \big)
        \geq \tau \|f\|^2 
        =\tau
        >0,
    \end{align*}
    where the second equality follows since $f \in \langle\1\rangle^{\perp}$ and $\E_\zeta [f] \in \langle\1\rangle$.     
    Therefore $\zeta\mapsto\alpha_{\gPI}(\zeta)$ is continuous on $\{\zeta\in\P(\Z):\zeta(z)\geq \tau,\, \forall z\in Z\}$. Since $\tau>0$ is arbitrary it follows that $\zeta\mapsto\alpha_{\gPI}(\zeta)$ is continuous on $\P_+(\Z)$.

   Finally, we prove \eqref{eq:gPI-char:eps}. 
   Let $\zeta, \zeta_\e, \delta_\e$ be given. We will prove that
   \begin{align} \label{pfzi}
       \liminf_{\e \to 0} \inf_{\substack{\varphi \in \cD_{\zeta_\e}(\Z) \\ 0 < \| \varphi - \1 \|_{\zeta_\e} < \delta_\e }}
\frac{\RF_{\zeta_\e}(\varphi)}{\RelEnt_{\zeta_\e}(\varphi)}
&\geq 2\alpha_{\gPI}(\zeta), \\\label{pfzh}
        \limsup_{\e \to 0} \inf_{\substack{\varphi \in \cD_{\zeta_\e}(\Z) \\ 0 < \| \varphi - \1 \|_{\zeta_\e} < \delta_\e }}
\frac{\RF_{\zeta_\e}(\varphi)}{\RelEnt_{\zeta_\e}(\varphi)}
&\leq 2\alpha_{\gPI}(\zeta),
   \end{align}
   which together establish \eqref{eq:gPI-char:eps}.
   
   We start with proving \eqref{pfzi}. Recall
   $
     \zeta_* = \min_{z \in \Z} \zeta(z) > 0
   $
   and take $\e$ small enough such that 
$(\zeta_\e)_* > \frac12 \zeta_*$ and $\delta_\e \leq \frac14 \sqrt{2 \zeta_*}$ (this upper bound is chosen a posteriori). Let $\varphi$ be any admissible function in \eqref{eq:gPI-char:eps} (note that $\varphi$ and the following objects depend on $\e$). Let 
    \begin{equation*}
      \delta \coloneqq \|\varphi-\1\|_{\zeta_\e} \in (0, \delta_\e), \ \ 
        f \coloneqq \frac{\varphi-\1}{\delta} \in \R^{|\Z|}.
    \end{equation*}
    By construction, $\varphi = \1 + \delta f$, $\E_{\zeta_\e}[f] = 0$ (i.e.\ $f \in \langle\1\rangle^{\perp_{\zeta_\e}}$) and $\Var_{\zeta_\e}(f) = \|f\|_{\zeta_\e}^2 = 1$, where the latter follows from the definition of $\delta$.     
    Using $\delta < \delta_\e \leq \frac14 \sqrt{2\zeta_*}$ and 
    \begin{equation} \label{pfzg}
  1 
  = \|f\|_{\zeta_\e}^2
  = \sum_{z \in \Z} f^2(z) {\zeta_\e}(z)
  \geq (\zeta_\e)_* \|f\|_\infty^2
  \geq \frac{\zeta_*}2 \|f\|_\infty^2,
\end{equation}
we obtain $\delta \|f\|_\infty < \frac12$. Hence, the expansions of the generalised Fisher information and relative entropy in Proposition~\ref{prop:GenDirFI-prop} and \eqref{eq:genRE-lin} apply. In particular, since $\|f\|_\infty^2 \leq \frac{2}{\zeta_*}$, the corresponding remainder terms in~\eqref{eq:genRF-lin},~\eqref{eq:genRE-lin} are bounded by some constants $C, C' > 0$ independent of $\e$, which is crucial for what follows. 
With this we obtain (recall that $\Var_{\zeta_\e}(f) = 1$)
    \begin{align} \label{pfzm}
       \frac{\RF_{\zeta_\e}(\varphi)}{\RelEnt_{\zeta_\e}(\varphi)} 
       = \frac{\RF_{\zeta_\e}(\1 + \delta f)}{\RelEnt_{\zeta_\e}(\1 + \delta f)}
       \geq \frac{2\calE_{\zeta_\e} (f ) - C \delta }{1 + C' \delta }
       \geq \frac{2\calE_{\zeta_\e} (f ) }{1 + C' \delta_\e } - C \delta_\e. 
    \end{align}    
Then, using that $\varphi$ is arbitrary, we obtain
\begin{align} \label{pfzn}
       \inf_{\substack{\varphi \in \cD_{\zeta_\e}(\Z) \\ 0 < \| \varphi - \1 \|_{\zeta_\e} < \delta_\e }}
\frac{\RF_{\zeta_\e}(\varphi)}{\RelEnt_{\zeta_\e}(\varphi)}
       \geq \frac2{1 + C' \delta_\e} \inf_{\substack{f\in \langle \1 \rangle^{\perp_{\zeta_\e}} \\ \|f\|_{\zeta_\e}=1}} \calE_{\zeta_\e} (f ) - C \delta_\e.
    \end{align}
    By the first characterisation of $\alpha_{\gPI}(\zeta)$ we recognise the minimization in the right-hand side as $\alpha_{\gPI}(\zeta_\e)$. Then, \eqref{pfzi} follows from $\delta_\e \to 0$ and the continuity of $\alpha_{\gPI}$.
    
    We now prove \eqref{pfzh}. While most of the proof is the same as that for \eqref{pfzi}, we are going to establish \eqref{pfzm} and \eqref{pfzn} in the reversed order with reversed inequalities. We choose $\e$ small enough such that $(\zeta_\e)_* > \frac12 \zeta_*$ and $\delta_\e \leq \delta_0$ for some $\e$-independent $\delta_0 > 0$ that we choose later by means of two upper bounds.
    Let $\delta \in (0, \delta_\e)$ and let $f$ be as in the infimum in \eqref{pfzn}, i.e.\ 
    $f\in \langle \1 \rangle^{\perp_{\zeta_\e}}$ and $\|f\|_{\zeta_\e}=1$. Take $\varphi \coloneqq \1+\delta f$, and note that $\varphi$ belongs to the class from the infimum in \eqref{pfzn}. Since \eqref{pfzg} still holds, we have for $\delta_0 \leq \frac14 \sqrt{2 \zeta_*}$ that Proposition~\ref{prop:GenDirFI-prop} and \eqref{eq:genRE-lin} apply with $f, \delta$. This yields the reverse of \eqref{pfzm} given by
    \begin{align*}
       \frac{\RF_{\zeta_\e}(\varphi)}{\RelEnt_{\zeta_\e}(\varphi)} 
       \leq \frac{2\calE_{\zeta_\e} (f ) }{1 - C' \delta_\e } + C \delta_\e
    \end{align*}
for some constant $C, C' > 0$ independent of $\e$, where we have assumed $\delta_0 \leq \frac12 (C')^{-1}$ to ensure that the denominator is positive. Moving the constants involving $\delta_\e$ to the left-hand side and taking the infimum over $f$ on both sides, we get
\begin{align*}
       (1 - C' \delta_\e) \Bigg( \inf_{\substack{\varphi \in \cD_{\zeta_\e}(\Z) \\ 0 < \|\varphi - \1\|_{\zeta_\e} < \delta_\e}} \frac{\RF_{\zeta_\e}(\varphi)}{\RelEnt_{\zeta_\e}(\varphi)} - C \delta_\e \Bigg)
       \leq 2 \inf_{\substack{f\in \langle \1 \rangle^{\perp_{\zeta_\e}} \\ \|f\|_{\zeta_\e}=1}} \calE_{\zeta_\e} (f )
       = 2 \alpha_{\gPI}(\zeta_\e).
\end{align*}
Taking the limsup over $\e$ yields \eqref{pfzh}.
\end{proof}

\section{Properties of the generalised PI and LSI constants}\label{sec:properties}
We now establish the main properties of the generalised constants $\alpha_{\gPI}(\zeta)$ and $\alpha_{\gLSI}(\zeta)$ defined in \eqref{def:PI-LSI-con}. We begin with two elementary observations: first, that $2 \alpha_{\gPI}(\zeta) \geq \alpha_{\gLSI}(\zeta)$; and second, that both the gPI and gLSI tensorise, meaning that a product Markov chain inherits the smallest gPI and gLSI constants among its components.
We then show that, in addition to $\zeta\mapsto \alpha_{\gPI}(\zeta)$, also $\zeta\mapsto \alpha_{\gLSI}(\zeta)$ is continuous on $\P_+(\Z)$. Finally, we derive positive lower bounds on $\alpha_{\gLSI}(\zeta)$, which then, by the inequality $2 \alpha_{\gPI}(\zeta) \geq \alpha_{\gLSI}(\zeta)$, also imply the same bounds for $\alpha_{\gPI}(\zeta)$.

\subsection{gLSI implies gPI and tensorisation}\label{sec:properties1}

The following result shows that $\alpha_{\gLSI}(\zeta) \leq 2 \alpha_{\gPI}(\zeta)$, which in particular implies that the gLSI implies the gPI. Note that this is also a property of the corresponding classical variants~\cite[Section~3]{bobkovTetali06}. 

\begin{prop}\label{prop:LSI-implies-Poincare}
    For any $\zeta\in\P_+(\Z)$, $\alpha_{\gLSI}(\zeta) \leq 2 \alpha_{\gPI}(\zeta)$. 
\end{prop}

The proof is a direct consequence of the characterisation of $\alpha_{\gPI}$ in~\eqref{eq:gPI-char:eps}; in particular, taking $\zeta_\e = \zeta$, removing the condition $\|\varphi-\1\|_{\zeta_\e}<\delta_\e$ from~\eqref{eq:gPI-char} and using~\eqref{def:PI-LSI-con} we arrive at the required result. 

The following result presents the tensorisation property of the generalised PI and LSI inequalities which carries over from the classical counterparts~\cite[Lemma~3.2]{diaconisSaloffCoste96}. 
\begin{prop}\label{prop:tensor}
    Let $d \geq 1$. For each $i\in \{1,\ldots,d\}$ let $\Z_i$ be a finite state space, $M_i$ be an irreducible  generator on $\Z_i$ and $\zeta_i\in \P_+(\Z_i)$. Define 
    \begin{equation*}
    \displaystyle
        \Z \coloneqq\prod_{i=1}^d \Z_i, \quad
        \zeta \coloneqq\zeta_1\otimes\ldots\otimes \zeta_d \in \P_+(\Z), \quad
        M \coloneqq\frac1d\sum_{i=1}^d \underbrace{I\otimes \ldots I}_{i-1} \otimes M_i \otimes \underbrace{I\otimes\ldots\otimes I}_{d-i} \in \R^{|\Z| \times |\Z|}.
    \end{equation*}
    Then, $M$ is an irreducible generator on $\Z$ and $\zeta$ satisfies the gPI and gLSI with constants 
    \begin{equation*}
        \alpha_{\gPI}(\zeta, M) =\frac1d \min_{1 \leq i \leq d} \alpha_{\gPI}(\zeta_i, M_i), \ \ \alpha_{\gLSI}(\zeta, M) =\frac1d \min_{1 \leq i \leq d}\alpha_{\gLSI} (\zeta_i, M_i).
    \end{equation*}
\end{prop}
\begin{proof} 
The proof for the gPI inequality follows exactly as in the classical case, see~\cite[Lemma~3.2]{diaconisSaloffCoste96} for instance. 
    For gLSI we provide a proof for $d=2$, and note that the calculations easily generalise to $d>2$ (albeit with messier notation). In the following $z=(x,y)$ and $z'=(x',y')$.
    Using $\Phi(r,s)=r(\frac{s}{r}-1-\log\frac{s}{r})$ and $\varphi^x(\cdot)=\varphi(x,\cdot)$, $\varphi^y(\cdot)=\varphi(\cdot,y)$ for $\varphi \in \cD_\zeta(\Z)$ 
    \begin{align}
        \RF_\zeta(\varphi) &= \sum_{(x,y)\in\Z} \zeta_1(x)\zeta_2(y)\sum_{(x',y')\in\Z} M\big((x,y),(x',y')\big)\Phi\big(\varphi(x,y),\varphi(x',y')\big)\notag\\
        & = \frac12\sum_{y\in \Z_2} \zeta_2(y)\sum_{x\in\Z_1}\zeta_1(x)\sum_{x'\in\Z_1}  M_1(x,x')\Phi\big(\varphi(x,y),\varphi(x',y)\big) \notag\\
        &\qquad\qquad + \frac12\sum_{x\in \Z_1} \zeta_1(x)\sum_{y\in\Z_2}\zeta_2(y)\sum_{y'\in\Z_2}  M_2(y,y')\Phi\big(\varphi(x,y),\varphi(x,y')\big) \notag\\
        &= \frac12\E_{\zeta_2}\big(\RF_{\zeta_1}(\varphi^y)\big) +  \frac12\E_{\zeta_1}\big(\RF_{\zeta_2}(\varphi^x)\big). \label{eq:Ten-LSI-FI}
    \end{align}
    Here and in the rest of this proof we use the convention that $\E_{\zeta_1}$, $\RF_{\zeta_1}$ and $\RelEnt_{\zeta_1}$ only act on $x$-variable, while similar operations with respect to $\zeta_2$ only act on the $y$-variable.  

    For any $\varphi\in \cD_\zeta(\Z)$, we define the marginal density $g\in \cD_{\zeta_2}(\Z_2)$ and family of conditional densities $\Psi_y\in \cD_{\zeta_1}(\Z_1)$ for $y\in \Z_2$ as
    \begin{equation*}
        g(y) \coloneqq\E_{\zeta_1}(\varphi^y)=\sum_{x\in \Z_1}\zeta_1(x)\varphi(x,y) > 0, \ \ \Psi_y(x) \coloneqq\frac{\varphi(x,y)}{g(y)},
    \end{equation*}
    which leads to the decomposition 
    \begin{equation*}
        \forall \varphi\in \cD_\zeta(\Z): \ \ \varphi(x,y)=g(y)\Psi_y(x).
    \end{equation*}
    Therefore for any $\varphi\in \cD_\zeta(\Z)$ we find 
    \begin{align*}
        \RelEnt_\zeta(\varphi)= \sum_{y\in\Z_2}\zeta_2(y)\sum_{x\in\Z_1} \zeta_1(x)g(y)\Psi_y(x)\bigl[\log \Psi_y(x)+\log g(y)\bigr]=\E_{\zeta_2}\Big(g(\cdot)\RelEnt_{\zeta_1}(\Psi_{(\cdot)})\Big) +\RelEnt_{\zeta_2}(g) ,
    \end{align*}
    where the second equality follows since $\E_{\zeta_1}(\Psi_y)=1$. 
    Using the notation $\alpha_{\gLSI}^j \coloneqq \alpha_{\gLSI}(\zeta_j,M_j)$ for $j = 1,2$,
    we therefore find 
    \begin{align}
        \min\{\alpha^1_{\gLSI},\alpha^2_{\gLSI}\} \, \RelEnt_\zeta(\varphi)
        &\leq \alpha^1_{\gLSI} \, \E_{\zeta_2}\Big(g(\cdot)\RelEnt_{\zeta_1}(\Psi_{(\cdot)})\Big)+ \alpha^2_{\gLSI}\,\RelEnt_{\zeta_2}(g) \notag\\
        & \leq \E_{\zeta_2}\Big(g(\cdot)\RF_{\zeta_1}(\Psi_{(\cdot)})\Big)+
         \RF_{\zeta_2}(g)\notag \\
        &=\E_{\zeta_2}\big(\RF_{\zeta_1}(\varphi^y)\big)+
         \RF_{\zeta_2}(g), \label{eq:Ten-LSI-Ent}
    \end{align}
    where the final equality follows since the generalised Fisher information is $1$-homogenous, i.e. $\RF_{\zeta_1}(c\xi)=c\RF_{\zeta_1}(\xi)$ for any $c\in \R$ and $\xi\in \cD_{\zeta_1}(\Z_1)$.

    Since $(r,s)\mapsto \Phi(r,s)$ is convex for $r,s>0$, using Jensen's inequality along with $g(y)=\E_{\zeta_1}(\varphi^y)$, we have the bound
     \begin{align*}
        \RF_{\zeta_2}(g) 
        &= \sum_{y,y'\in\Z_2} \zeta_2(y)M_2(y,y') \Phi\big(\E_{\zeta_1}(\varphi^y),\E_{\zeta_1}(\varphi^{y'})\big) \\
        &\leq \sum_{y,y'\in\Z_2} \zeta_2(y)M_2(y,y') \E_{\zeta_1} \Phi \big(\varphi^y,\varphi^{y'}\big) \\
        & = \sum_{x \in \Z_1} \big(\RF_{\zeta_2}(\varphi^x)\big) \zeta_1(x) = \E_{\zeta_1}\big(\RF_{\zeta_2}(\varphi^x)\big).
    \end{align*}

    Substituting this bound into~\eqref{eq:Ten-LSI-Ent} and using~\eqref{eq:Ten-LSI-FI} we find
    \begin{align*}
        \min\{\alpha^1_{\gLSI},\alpha^2_{\gLSI}\}\,\RelEnt_{\zeta}(\varphi) \leq 2 \RF_\zeta(\varphi),
    \end{align*}
    which implies that $2\alpha_{\gLSI}\geq \min\{\alpha^1_{\gLSI},\alpha^2_{\gLSI}\}$. 
    
    To obtain the opposite inequality, we 
    repeat the arguments above with the simple choices $\varphi(x,y)=\xi(x)$ and $\varphi(x,y)=\eta(y)$. For the first choice, this yields $\RelEnt_\zeta(\varphi) = \RelEnt_{\zeta_1}(\xi)$ and $\RF_\zeta(\varphi) = \RF_{\zeta_1}(\xi)$, and thus $2\alpha_{\gLSI}(\zeta,M) \leq \alpha_{\gLSI}(\zeta_1, M_1)$. Similarly, for the second choice of $\varphi$ we get $2\alpha_{\gLSI}(\zeta,M) \leq \alpha_{\gLSI}(\zeta_2, M_2)$. 
\end{proof}

\subsection{Continuity of the generalised LSI constant}\label{sec:continuity}

With the continuity of the generalised PI constant already established in Proposition \ref{prop:gPI-char}, 
we now prove that $\alpha_{\gLSI}$ is continuous as well. The proof is based on the variational definition of  $\alpha_{\gLSI}(\zeta)$ as the infimum of $\RF_\zeta(\varphi) / \RelEnt_\zeta(\varphi)$ over $\varphi$, see \eqref{def:PI-LSI-con}. 
If $\RF_\zeta(\varphi) / \RelEnt_\zeta(\varphi)$ were to be continuous in $\varphi \in \cD_\zeta(\Z)$, then the proof would be much simpler. However, $\RF_\zeta(\varphi) / \RelEnt_\zeta(\varphi)$ may be discontinuous at $\varphi = \1$; see Example \ref{ex:disct} below. 
To obtain continuity of $\alpha_{\gLSI}$, we establish $\Gamma$-convergence of $\RF_\zeta(\varphi) / \RelEnt_\zeta(\varphi)$, see \eqref{eq:Gamma-convergence}, which implies the convergence of minima and thus the continuity of $\alpha_{\gLSI}$.

\begin{example}[Discontinuity of $\RF_\zeta(\varphi) / \RelEnt_\zeta(\varphi)$ at $\varphi = \1$]
    \label{ex:disct}
Let $\Z = \mathbb Z / (3 \mathbb Z) \cong \{0,1,2\}$. We treat $i \in \Z$ as indices. We consider the cyclic chain given by $M_{i,i+1} = 1$, $M_{ii} = -1$ and $0$ otherwise. In vector notation, let
\begin{align*}
    \zeta = \frac14 \begin{bmatrix}
        2 \\ 1 \\ 1
    \end{bmatrix}, \quad
    f^1 = \begin{bmatrix}
        0 \\ 1 \\ -1
    \end{bmatrix}, \quad
    f^2 = \begin{bmatrix}
        1 \\ -2 \\ 0
    \end{bmatrix}, \quad
    \varphi^k = \1 + \delta f^k \ \ (k = 1,2),
\end{align*}
where we will later pass $\delta \to 0$. By construction, $\mathbb E_\zeta [f^k] = 0$ for $k=1,2$. Then, for $\delta$ small enough, $\varphi^k \in \cD_\zeta(\Z)$, and the expansions in \eqref{eq:genRF-lin} and \eqref{eq:genRE-lin} apply and yield
\[
  \RF_\zeta(\varphi^k) = \delta^2 \calE_\zeta(f^k) + \cO (\delta^3), \quad
\RelEnt_\zeta(\varphi^k) = \frac{\delta^2 }{2}\Var_\zeta(f^k) + \cO (\delta^3). 
\]
Straightforward computations yield
\begin{gather*}
    \Var_\zeta(f^k)
    = \sum_{i=0}^2 (f_i^k)^2 \zeta_i
    = \left\{ \begin{aligned}
        &\frac12 && \text{if } k = 1 \\
        &\frac32 && \text{if } k = 2,
    \end{aligned} \right.
    \\
    \calE_\zeta(f^k)
    = \frac12 \sum_{i,j=0}^2 M_{ij} \zeta_i (f_i^k - f_j^k)^2
    = \frac12 \sum_{i=0}^2 \zeta_i (f_i^k - f_{i+1}^k)^2
    = \left\{ \begin{aligned}
        &\frac78 && \text{if } k = 1 \\
        &\frac{23}8 && \text{if } k = 2,
    \end{aligned} \right.
\end{gather*}
and therefore $\RF_\zeta(\varphi) / \RelEnt_\zeta(\varphi)$ is not continuous at $\varphi = \1$ since 
\[
  \lim_{\delta \to 0}
  \frac{ \RF_\zeta(\1 + \delta f^k) }{ \RelEnt_\zeta(\1 + \delta f^k) }
  = \frac{ 2 \calE_\zeta(f^k) }{ \Var_\zeta(f^k) }
  = \left\{ \begin{aligned}
        &\frac{21}{12} && \text{if } k = 1, \\
        &\frac{23}{12} && \text{if } k = 2.
    \end{aligned} \right.
\]
\end{example}

\begin{theorem}
    \label{thm:continuity-alpha-LSI}
    Let $M$ be a generator of an irreducible Markov chain on a finite state space $\Z$ and $\zeta\in \P_+(\Z)$, the set of positive probability measures on $\Z$. The generalised LSI constant $\alpha_{\gLSI}(\zeta)$, see \eqref{def:PI-LSI-con}, is continuous on $\P_+(\Z)$.
\end{theorem}

\begin{proof}
    We show that for any $\zeta \in \P_+(\Z)$ and any sequence $\{\zeta_\varepsilon\}_{\eps>0} \in \P_+(\Z)$ with $\zeta_\varepsilon \rightarrow \zeta$ as $\varepsilon \rightarrow 0$ we have $\alpha_{\gLSI}(\zeta_\varepsilon) \rightarrow \alpha_{\gLSI}(\zeta)$ as $\varepsilon \rightarrow 0$. Let such  $\zeta, \zeta_\e$ be given. 
    We introduce the corresponding functionals $E_\varepsilon, E_0 : \R_{>0}^{|\Z|} \rightarrow \R$ as 
    \begin{align*}
        E_\varepsilon(\varphi) &\coloneqq \begin{dcases}
            \frac{\RF_{\zeta_\varepsilon}(\varphi)}{\RelEnt_{\zeta_\varepsilon}(\varphi)}, & \text{if } \varphi \in \mathcal{D}_{\zeta_\varepsilon}(\Z) \setminus \{\1\}, \\
            +\infty, & \text{otherwise,}
        \end{dcases}
        \\
        E_0(\varphi) &\coloneqq \begin{dcases}
            \frac{\RF_{\zeta}(\varphi)}{\RelEnt_{\zeta}(\varphi)}, & \text{if } \varphi \in \mathcal{D}_{\zeta}(\Z) \setminus \{\1\}, \\
            2\alpha_{\gPI}(\zeta), &
            \text{if } \varphi = \1, \\
            +\infty, & \text{otherwise}.
        \end{dcases}
    \end{align*}
    In the following, we will prove the stronger result 
    \begin{equation}\label{eq:Gamma-convergence}
    \operatorname*{\Gamma-lim}_{\varepsilon \rightarrow 0} E_\varepsilon = E_0,
    \end{equation}
    i.e.\ the functional $E_\eps$ $\Gamma$-converges 
    to the functional $E_0$ as $\eps\to 0$, which is equivalent to the following two conditions:
    \begin{enumerate}[label=($\Gamma$\arabic*)]
        \item For any $\varphi \in \R_{>0}^{|\Z|}$, and any sequence $\{\varphi_\varepsilon\}_{\eps>0} \in \R_{>0}^{|\Z|}$ which satisfies $\varphi_\varepsilon \rightarrow \varphi$ as $\varepsilon \rightarrow 0$ it holds
    \begin{equation}\label{eq:liminf-ineq}
        \liminf_{\varepsilon \rightarrow 0} E_\varepsilon(\varphi_\varepsilon) \geq E_0(\varphi).
    \end{equation}
    \item For any $\varphi \in \R_{>0}^{|\Z|}$ there exists a sequence $\{\varphi_\varepsilon\}_{\eps>0} \in \R_{>0}^{|\Z|}$ such that $\varphi_\varepsilon \rightarrow \varphi$ as $\varepsilon \rightarrow 0$ and it holds
    \begin{equation}\label{eq:limsup-ineq}
        \limsup_{\varepsilon \rightarrow 0} E_\varepsilon(\varphi_\varepsilon) \leq E_0(\varphi).
    \end{equation}
    \end{enumerate} 
    This $\Gamma$-convergence result along with the observation that $E_\e, E_0$ take finite values at some  
    $\varphi \in \R_{>0}^{|\Z|}$ ensure that the infima converge, i.e.\ $\alpha_{\gLSI}(\zeta_\varepsilon) \rightarrow \alpha_{\gLSI}(\zeta)$ as $\varepsilon \rightarrow 0$, see \cite{braides02},      
    which is the claimed result. \comm{[pf p.51 below]}

    We first point out that \eqref{eq:liminf-ineq},~\eqref{eq:limsup-ineq} hold if $\varphi \neq \1$.  This follows since $(\eta,\varphi)\mapsto \RelEnt_\eta(\varphi),\RF_\eta(\varphi)$ are continuous maps on $\P_+(\Z)\times \R^{|\Z|}_{>0}$ and $\RelEnt_{\eta}(\varphi) > 0$ for all $\varphi \in \cD_\eta(\Z) \setminus \{\1\}$, see Proposition~\ref{prop:GenDirFI-prop}.
    Consequently, for any $\varphi\in \cD_\zeta(\Z)\setminus\{\1\}$ and any sequence $\{\varphi^\eps\}_{\eps>0}\in \cD_{\zeta_\eps}(\Z)$ with $\varphi^\eps\to \varphi$, we have that $\| \varphi^\eps - \1 \| \geq \frac12 \| \varphi - \1 \| > 0$ for all $\eps>0$ small enough, and then
    \begin{align*}
        \lim_{\eps \to 0} E_\eps(\varphi_\varepsilon) = E_0(\varphi).
    \end{align*}
    Note that one example is given by $\varphi^\eps=\varphi \frac{\zeta}{\zeta_\eps}$, which proves the existence statement in~\eqref{eq:limsup-ineq}. 
    
    We now prove~\eqref{eq:liminf-ineq} for $\varphi=\1$. We denote the left-hand side in \eqref{eq:liminf-ineq} by $a \in [0,\infty]$. We may assume that $a < \infty$, as otherwise \eqref{eq:liminf-ineq} is trivially satisfied. Then, upon extracting a subsequence of $\e$ (not relabeled), we may assume that $\infty > E_\e(\varphi_\e) \to a$ as $\e \to 0$. Finiteness implies that $\{\varphi_\eps\}_{\eps>0}\subset \cD_{\zeta_\eps}(\Z)$ satisfies $\varphi_\eps\neq \1$ for all $\e$. We introduce $h_\eps = \varphi_\eps-\1$, and recall from the proof of Proposition \ref{prop:gPI-char} that $0 < \| h_\e \|_{\zeta_\e} \to 0$ as $\e \to 0$. Then, \eqref{eq:liminf-ineq} follows from \eqref{eq:gPI-char:eps} by
    \begin{align*}
        \lim_{\eps\to 0} E_\e(\varphi_\eps)
        =
        \lim_{\eps\to 0} \frac{\RF_{\zeta_\eps}(\varphi_\eps)}{\RelEnt_{\zeta_\eps}(\varphi_\eps)}
        \geq 
        \lim_{\eps\to 0} \inf_{\substack{\varphi \in \cD_{\zeta_\e}(\Z) \\ 0 < \| \varphi - \1 \|_{\zeta_\e} < \| h_\e \|_{\zeta_\e} }} \frac{\RF_{\zeta_\eps}(\varphi)}{\RelEnt_{\zeta_\e}(\varphi)}
        = 2\alpha_{\gPI}(\zeta) = E_0(\1).
    \end{align*}
    
    Next, we show  \eqref{eq:limsup-ineq} for $\varphi = \1$. We use Proposition \ref{prop:gPI-char} to characterise
    \begin{equation*}
        E_0(\1) = \lim_{\eps\to 0} \inf_{\substack{\varphi \in \cD_{\zeta_\e}(\Z) \\ 0 < \| \varphi - \1 \|_{\zeta_\e} < \e }} \frac{\RF_{\zeta_\eps}(\varphi)}{\RelEnt_{\zeta_\e}(\varphi)}.
    \end{equation*}
    Let $\varphi_\e$ be a minimising sequence for the sequence of minimisation problems in the right-hand side. From the definition of $E_\e$ we then obtain that $E_0(\1)= \lim_{\eps\to 0} E_\e(\varphi_\e)$, and thus it is left to verify that $\varphi_\e \to \1$ as $\e \to 0$. This follows from $\| \varphi - \1 \|_{\zeta_\e} < \e$; see \eqref{pfzg} for details.
\end{proof}

The following remark discusses the additional continuity of the gPI and gLSI constants with the generator matrix as an additional variable. 
\begin{rem}\label{rem:continuity-gLSI-with-generator}
   So far, when dealing with generalised functional inequalities, we have suppressed the explicit dependence on the generator in the analysis by keeping it fixed. 
   By applying similar arguments as in the proofs of Proposition~\ref{prop:gPI-char} and Theorem~\ref{thm:continuity-alpha-LSI} it can be shown that the maps \[(\zeta,M)\mapsto \alpha_{\gPI}(\zeta,M), \, \alpha_{\gLSI}(\zeta,M)\] 
   are continuous over $\P_+(\Z)\times \mathcal M(\Z)$, where $\mathcal M(\Z)$ is the space of irreducible generators on $\Z$, which is a convex cone in $\R^{|\Z| \times |\Z|}$.
   To see this it is sufficient to note that $\calE_{\zeta,M}(f)$ and $\RF_{\zeta,M}(\varphi)$ depend linearly on $M$. 
\end{rem}

\subsection[Lower bound on gLSI constant]{Lower bounds on $\alpha_{\gPI}(\zeta)$ and $\alpha_{\gLSI}(\zeta)$}
\label{s:al:gLSI:LB}
In the classical setting where the steady state $\pi$ is the reference measure, \cite{Saloff-Coste97} provides implicit lower bounds on the PI and LSI constants. More precisely, \cite[Section 3.2]{Saloff-Coste97} establishes various implicit lower bounds on $\alpha_{\PI}$ and \cite[Theorem 2.2.3]{Saloff-Coste97} provides an implicit lower bound for the constant appearing in the standard LSI~\eqref{eq:altLSI-intro}. This also provides a lower bound for the constant appearing in the classical LSI inequality~\eqref{def:cLSI}, see~\cite[Proposition 3.6]{bobkovTetali06} for the reversible case and~\cite[Section 1]{Miclo18} for the general case. In this section, we present explicit lower bounds on the \textit{generalised} PI and LSI constants for \textit{any} reference measure $\zeta$. 

In Theorem~\ref{thm:pos} below, we establish the positivity of the gLSI constant, which, in turn, by Proposition~\ref{prop:LSI-implies-Poincare}, implies the positivity of the gPI constant as well. The proof of this main result requires the following lemma, which collects several properties of a symmetrised generator and makes use of the standard LSI (recall~\eqref{eq:altLSI-intro}), which involves the centred entropy (recall $\Ent_\pi(\cdot)$ introduced in Remark~\ref{rem:cenEnt}).

Given $\zeta\in\P_+(\Z)$ we define the $\zeta$-symmetrised generator
$M^\zeta\in \R^{|\Z|\times |\Z|}$ as 
\begin{equation}\label{eq:symGen}
   \forall z'\neq z: \ \ M^\zeta(z,z')\coloneqq
    \frac12 \Bigl(M(z,z')+\frac{\zeta(z')}{\zeta(z)} M(z',z)\Bigr); \ \ M^\zeta(z,z)\coloneqq
    -\sum_{\substack{z'\in \Z \\ z'\neq z}} M^\zeta(z,z').
\end{equation}
Note that the second term in the definition of $M^\zeta$ is precisely the adjoint operator to $M$ in $L^2(\zeta)$, which clarifies that this is a symmetrised generator. 
We also introduce the Dirichlet form using the generator $M^\zeta$ as
\begin{equation}\label{eq:symGenDir}
    \calE_\zeta(f,M^\zeta) = \frac12 \sum_{z,z'\in \Z} M^\zeta(z,z')\zeta(z) \big(f(z)-f(z')\big)^2
    = \calE_\zeta(f,M). 
\end{equation}
The following lemma then establishes important properties of $M^\zeta$ and provides estimates relating its spectral gap to the standard LSI constant of $M^\zeta$ introduced in Remark \ref{rem:LitLSI-PI}.
\begin{lem}\label{lem:symGen}
    Let $M$ be a generator of an irreducible Markov chain on a finite state space $\Z$ and $\zeta\in \P_+(\Z)$. 
    The symmetrised generator $M^\zeta$~\eqref{eq:symGen} is irreducible and is reversible with respect to its steady state $\zeta$. Consequently, $\zeta$ satisfies the PI and the standard LSI, i.e., 
    \begin{equation}\label{eq:altLSI}
        \forall f\in \R^{|\Z|}_{>0}: \ \ \Var_\zeta(f)\leq \frac{1}{\lambda_{\zeta}}\calE_\zeta(f,M^\zeta), \ \ \Ent_\zeta(f^2) \leq \frac{1}{\alpha_{\sLSI} (M^\zeta)}\calE_\zeta(f,M^\zeta),
    \end{equation}
    where $\Ent_\zeta(\cdot)$ is the centred entropy~\eqref{eq:CentEnt},  $\calE_\zeta(\cdot,M^\zeta)$ is the Dirichlet form using the generator $M^\zeta$, and $\lambda_{\zeta} = \alpha_{\PI}(M^\zeta) = \alpha_{\gPI}(\zeta, M) > 0$ is the spectral gap for the generator $-M^\zeta$. Furthermore, we have the relation
    \begin{equation}\label{eq:altLSI-bound}
        \frac{\lambda_{\zeta}}{2+\log\dfrac{1}{\zeta_*}}\leq \alpha_{\sLSI}(M^\zeta) \leq 2\lambda_{\zeta},
    \end{equation}
    where $\zeta_*=\min_z\zeta(z)$.
\end{lem}

\begin{proof}
    The irreducibility of $M^\zeta$ follows from the irreducibility of $M$. Using the characterisation 
    \begin{equation*}
        (M^\zeta f)(z)= \sum_{\substack{z'\in\Z\\ z'\neq z}}
        M^\zeta(z,z')\bigl(f(z')-f(z)\bigr), \ \ \forall f\in \R^{|\Z|},
    \end{equation*}
    it follows that $(M^\zeta)^T\zeta=0$ (i.e.\ $\zeta$ is the steady state),  since for any $f\in \R^{|\Z|}$ we have
    \begin{equation*}
        (M^\zeta f,\1)_\zeta=\frac12\sum_{z,z'\in\Z}\Big(M(z,z')\zeta(z)+\zeta(z')M(z',z)\Big)\bigl(f(z')-f(z)\bigr)=0.
    \end{equation*}
    The final equality follows by exchanging the indices in the second sum.
    Finally, the reversibility with respect to $\zeta$ follows since for $z\neq z'$
    \begin{equation*}
        \zeta(z)M^\zeta(z,z')= \zeta(z')M^\zeta(z',z).
    \end{equation*}
    The standard LSI inequality~\eqref{eq:altLSI} follows from classical arguments for reversible Markov chains (see for instance~\cite[Section~3.1]{diaconisSaloffCoste96}), with $\alpha_{\sLSI}(M^\zeta)>0$ presented in~\cite[Theorem~2.2.3]{Saloff-Coste97}. The connection between the PI constant and the spectral gap is standard for reversible Markov chains~\cite{bobkovTetali06}.

    %I have added \sloppy to make sure reference stays within margin
    \sloppy{Now we prove the bounds on $\alpha_{\sLSI}(M^\zeta)$ in~\eqref{eq:altLSI-bound}. The upper bound is standard~\cite[Lemma~3.1]{diaconisSaloffCoste96} and follows exactly as in the proof of  Proposition~\ref{prop:LSI-implies-Poincare}. 
    The lower bound requires the Rothaus lemma~\cite[Lemma~5.1.4]{BakryGentilLedoux14}, which states that for any $g\in \R^{|\Z|}$ and $a\in \R$ we have}
    \begin{equation*}
        \Ent_\zeta\bigl((g+a)^2\bigr) \leq \Ent_\zeta\bigl(g^2\bigr) + 2\E_\zeta(g^2).
    \end{equation*}
    Using $g=f-\E_\zeta(f)$ and $a=\E_\zeta(f)$, the Rothaus lemma becomes
    \begin{equation}\label{eq:ModRothaus}
        \Ent_\zeta(f^2) \leq \Ent_\zeta\bigl((f-a)^2\bigr) + 2\Var_\zeta(f).
    \end{equation}
    Setting $m\coloneqq\E_\zeta((f-a)^2)=\Var_\zeta(f)$ we find
    \begin{align*}
            \Ent_\zeta\bigl((f-a)^2\bigr)=\E_\zeta\bigg((f-a)^2\log\frac{(f-a)^2}{m}\bigg) \leq m \log\frac{\|f-a\|_\infty^2}{m}.
    \end{align*}
    With $\zeta_*=\min_z\zeta(z)$ and
    \begin{equation*}  
           \|f-a\|_\infty^2
           \leq \|f-a\|^2
           \leq \frac{1}{\zeta_*} \|f-a\|_\zeta^2
           = \frac{m}{\zeta_*},
    \end{equation*}
    where the norms are defined in Section~\ref{sec:notation},
    we arrive at 
    \begin{equation*}
        \Ent_\zeta\bigl((f-a)^2\bigr) \leq m\log\frac{1}{\zeta_*} = \Var_\zeta(f)\log\frac{1}{\zeta_*}.
    \end{equation*}    
    Substituting into~\eqref{eq:ModRothaus} we find
    \begin{align*}
        \Ent_\zeta(f^2) \leq \biggl(2+\log\frac{1}{\zeta_*}\biggr) \Var_\zeta(f) \leq \biggl(2+\log\frac{1}{\zeta_*}\biggr)\frac{1}{\lambda_{\zeta}}\calE_\zeta(f,M^\zeta),
    \end{align*}
    where the second inequality follows from the Poincar\'e inequality for $\zeta$ with respect to $M^\zeta$. This concludes our lower bound for $\alpha_{\sLSI}(M^\zeta)$. 
\end{proof}

Using the previous lemma, we now establish the main theorem of this section, which provides lower bounds on the generalised Poincar\'e and log-Sobolev constants.

\begin{theorem}\label{thm:pos}
    Let $M$ be a generator of an irreducible Markov chain on a finite state space $\Z$ and $\zeta\in \P_+(\Z)$. Let $\zeta_*=\min_z\zeta(z)$ and  $M_*=\min \{M(z,z')\,:\,z,z'\in \Z, \, M(z,z')>0 \}$ be the smallest positive rate in $M$. 
    \begin{enumerate}[label=(\arabic*)]
        \item  We have the lower bound
        \begin{equation} \label{eq:lowBound:gPI}
            \alpha_{\gPI}(\zeta)
  \geq \frac{ \zeta_* M_* }{|\Z|}.
        \end{equation}
        \item Using $\lambda_{\zeta}$ for the spectral gap of $-M^\zeta$ and $\alpha_{\sLSI}(M^\zeta)$ for the standard LSI constant (defined in Lemma~\ref{lem:symGen}) we have 
        \begin{equation}\label{eq:lowBound}
            \alpha_{\gLSI}(\zeta)
            \geq 2\alpha_{\sLSI}(M^\zeta)
            \geq \dfrac{\lambda_{\zeta}}{1 + \dfrac12 \log\dfrac{1}{\zeta_*}}
            = \dfrac{ \alpha_{\gPI}(\zeta) }{1 + \dfrac12 \log\dfrac{1}{\zeta_*}}
  \geq \dfrac{ \zeta_* M_* |\Z|^{-1} }{1 + \dfrac12 \log\dfrac{1}{\zeta_*}}.
        \end{equation}
    \end{enumerate}    
\end{theorem}

\begin{proof}
     We start with the proof of~\eqref{eq:lowBound:gPI}. Consider the definition of $\alpha_{\gPI}(\zeta)$ in \eqref{def:PI-LSI-con}, and take a corresponding function $f \in \R^{|\Z|}$ with $f \notin \lrang{\1}$. Let $\overline f = \max_{z \in \Z} f(z)$ and $\underline f = \min_{z \in \Z} f(z)$, and some corresponding maximiser $\overline z \in \Z$ and minimiser $\underline z \in \Z$.
Since $M$ is irreducible there exists a path $(z_i)_{i=0}^\ell \subset \Z$ of length $\ell+1 \leq |\Z|$ with $z_0 = \underline z$ and $z_\ell = \overline{z}$ such that $M(z_i,z_{i-1}) > 0$ for all $i = 1,\ldots,\ell$. Setting $\kappa_i \coloneqq f(z_i)-f(z_{i-1})$, we get $\overline f-\underline f = f(z_\ell) - f(z_0)= \sum_{i=1}^{\ell} \kappa_i$. Hence, by the Cauchy-Schwartz inequality,
\begin{align*}
     (\overline f-\underline f)^2 = \biggl(\sum_{i=1}^{\ell} \kappa_i\biggr)^2 \leq \ell \sum_{i=1}^{\ell}\kappa_i^2.
\end{align*}
Using this we obtain
\begin{equation*}
    2\calE_\zeta(f) \geq \sum_{i=1}^{\ell} M(z_i,z_{i-1})\zeta(z_i)(f(z_i)-f(z_{i-1}))^2 
    \geq  M_*\zeta_* \frac{ (\overline f - \underline f)^2}{\ell} 
    \geq \frac{1}{|\Z|-1}M_*\zeta_* (\overline f - \underline f)^2.
\end{equation*}
     
In addition, we bound the denominator in \eqref{def:PI-LSI-con} as
\begin{align*}
        \Var_\zeta(f) 
        = \frac12 \sum_{z,z' \in \Z} ( f(z) - f(z') )^2 \zeta(z) \zeta(z')
        \leq \frac12 (\overline f - \underline f)^2.
\end{align*} 
Then, \eqref{eq:lowBound:gPI} follows directly.

     Regarding \eqref{eq:lowBound}, the second inequality is \eqref{eq:altLSI-bound}, the last inequality is a direct application of \eqref{eq:lowBound:gPI} and the equality follows from $\alpha_{\gPI}(\zeta)=\alpha_{\PI}(M^\zeta)=\lambda_\zeta$. The first inequality requires more work. To prove it, let $\varphi$ be as in \eqref{def:PI-LSI-con}, i.e.\ $\varphi \in \cD_\zeta(\Z)$ and $\varphi \neq \1$. We will bound $\RF_\zeta(\varphi)$ from below and $\RelEnt_\zeta(\varphi)$ from above. Using the inequality $r-1-\log r \geq (\sqrt{r}-1)^2$ for $r>0$ we find \comm{[this ineq follows from $\log s \leq s-1$ with $s^2 = r$]}
     \begin{equation*} %\label{eq:RF-Dir}
         \RF_\zeta(\varphi) \geq  \sum_{z,z'\in\Z} M(z,z')\zeta(z)  \Bigl(\sqrt{\varphi(z')}-\sqrt{\varphi(z)}\Bigr)^2 = 2\calE_\zeta\bigl(\sqrt{\varphi}, M\bigr)
         = 2 \calE_\zeta(\sqrt{\varphi},M^\zeta), 
     \end{equation*}     
    where the Dirichlet form $\calE_\zeta(\cdot,M^\zeta)$ is defined in~\eqref{eq:symGenDir}. 
    Then, using the (alternative form of the) classical LSI inequality~\eqref{eq:altLSI}, for any density $\varphi\in\cD_\zeta(\Z)$ (recall~\eqref{def:SpaceDen}) we find 
    \begin{align*}
        \RelEnt_\zeta(\varphi)= \Ent_\zeta(\varphi) \leq \frac{1}{\alpha_{\sLSI} (M^\zeta)}\calE_\zeta\bigl(\sqrt{\varphi},M^\zeta\bigr) \leq \frac{1}{2\alpha_{\sLSI} (M^\zeta)}\RF_\zeta(\varphi),
    \end{align*}
    where $\Ent_\zeta$ is the centred entropy~\eqref{eq:CentEnt}. Therefore, we have 
    $
        \alpha_{\gLSI}(\zeta)\geq 2\alpha_{\sLSI}(M^\zeta)
    $. The second and final inequality in~\eqref{eq:lowBound} follow from Lemma~\ref{lem:symGen} and~\eqref{eq:lowBound:gPI} respectively.
\end{proof}

If one has more information on $M$, then sharper bounds than \eqref{eq:lowBound:gPI} are available; see e.g.\ \cite[Section 3.2]{Saloff-Coste97}. Furthermore, we also have the following straightforward upper bound for the gLSI constant
\begin{equation}\label{eq:upper-bound-gLSI}
    \alpha_{\gLSI}(\zeta)
    \leq 2\alpha_{\gPI}(\zeta)
    = 2\alpha_{\PI}(\zeta,M^\zeta)
    = 2\lambda_{\zeta}.
\end{equation}

We do not expect that the explicit lower bound~\eqref{eq:lowBound} is sharp. Yet, the following remark demonstrates that the lower bound has to be at least linear in $\zeta_*$. 
\begin{rem}
    Here we demonstrate by means of an example that, at least for certain $M$, the lower bounds on $\alpha_{\gPI}(\zeta)$ and $\alpha_{\gLSI}(\zeta)$ have to be at least linear in $\zeta_*$, i.e.\ $\alpha_{\gPI}(\zeta) + \alpha_{\gLSI}(\zeta) \leq C \zeta_*$. Similar to Example \ref{ex:disct}, but now for a larger state space,
    let $\Z = \mathbb Z / (2N \mathbb Z) \cong \{0,1,\ldots,2N-1\}$ with large $N$, and let $M$ be the cyclic chain given by $M_{i,i+1} = 1$, $M_{ii} = -1$ and $0$ otherwise. 
    Let 
\[
  \zeta_i = \left\{ \begin{aligned}
  &\frac\e N
  && i = N-1, 2N-1  \\
  &\frac1N
  &&\text{otherwise,}
\end{aligned} \right.
\qquad
f_i =  \left\{ \begin{aligned}
  &1
  && i < N  \\
  &-1
  && i \geq N,
  \end{aligned} \right.
  \qquad
  \varphi_i = \left\{ \begin{aligned}
  &\frac12
  && i < N  \\
  &\frac32
  && i \geq N.
  \end{aligned} \right.
\]
Note that $\mathbb E_\zeta[f] = 0$. Except for errors of size $O(\e, \frac1N)$, $\zeta$ and $\varphi$ are normalised. These errors do not play a significant role in the following computation, and thus we neglect them. 

Next we estimate $\alpha_{\gPI}(\zeta)$ and $\alpha_{\gLSI}(\zeta)$. First,
since $\mathbb E_\zeta[f] = 0$,
\[
  \Var_\zeta(\varphi)
  = \sum_{i} \zeta_i f_i^2 = 1 + O\Big( \frac1N \Big).
\]
Similarly, with $\Psi(x) = x\log x-x+1 \geq 0$, which vanishes only at $x = 1$,
\[
  \RelEnt_\zeta(\varphi)
  = \sum_{i} \zeta_i \Psi(\varphi_i)
  = \frac12 \Psi \Big( \frac12 \Big) + \frac12 \Psi \Big( \frac32 \Big) + O \Big(\e, \frac1N \Big),
\]
which is also close to a universal, positive constant. Second,
\begin{align*}
    \calE_\zeta(f)
  &= \frac12 \sum_{i \neq j} \zeta_i M_{ij} (f_i - f_j)^2
  = \frac12  \sum_{i = 0}^{2N-1} \zeta_i M_{i,i+1} (f_i - f_{i+1})^2 \\
  &= 2 \zeta_{N-1} + 2 \zeta_{2N-1} 
  = 4 \frac\e N
  = 4 \zeta_*.
\end{align*}
Similarly, with $\Phi(x) = x - 1 - \log x$, which vanishes only at $x = 1$,
\begin{align*}
    \RF_\zeta(\varphi)
  &= \sum_{i \neq j} \zeta_i \varphi_i M_{ij} \Phi \Big( \frac{\varphi_j}{\varphi_i} \Big)
  = \sum_{i = 0}^{2N-1} \zeta_i \varphi_i M_{i,i+1} \Phi \Big( \frac{\varphi_{i+1}}{\varphi_i} \Big) \\
  &=  \zeta_{N-1} \varphi_{N-1} \Phi (3) + \zeta_{2N-1} \varphi_{2N-1} \Phi \Big( \frac13 \Big) 
  = \frac\e{2N} \Big( 3 \Phi \Big( \frac13 \Big) + \Phi (3) \Big)
  = C \zeta_*.
\end{align*}
Thus, for $\e, \frac1N$ small enough, we obtain $\alpha_{\gPI}(\zeta) \leq C \zeta_*$ and $\alpha_{\gLSI}(\zeta) \leq C' \zeta_*$ for some universal constants $C,C' > 0$.
\end{rem}

\section{Applications}\label{sec:applications}

We now present two applications of our generalised functional inequalities. The first one deals with decay estimates for the distance between two solutions to the forward Kolmogorov equation~\eqref{eq:Nota-forKol} in variance and relative entropy, see Section \ref{sec:application-long-term-estimates}. The second one is qualitative error estimates in coarse-graining, which improve the results in \cite{HilderSharma24}. Here, we present results in the case of general generators, see Section \ref{sec:estimates-without-scale-sep}, and a generator with explicit scale separation, see Section \ref{sec:error-estimates-with-scale-sep}. The first result also provides a recipe to evaluate the quality of the chosen coarse-graining map, which is presented in Section \ref{sec:different-cg-maps}. 

\subsection{Comparing two solutions to the forward Kolmogorov equation}\label{sec:application-long-term-estimates}

As discussed in Section \ref{sec:def-gDF-gFI}, the generalised Dirichlet form and Fisher information are the dissipation of the generalised variance and entropy, respectively, along solutions of the forward Kolmogorov equation \eqref{eq:Nota-forKol}. In the classical case, this relation allows to estimate the rate of convergence to equilibrium in variance and entropy, respectively. Using the generalised Poincar\'e inequality and generalised log-Sobolev inequality, we can extend this to measure the convergence in time of one solution of \eqref{eq:Nota-forKol} to another.

We point out that such results comparing two different solutions are not entirely new in the literature: estimates in $L^2(\pi)$ can be derived using the spectral gap of the symmetrised generator with respect to $\pi$ by a straightforward time-derivative argument; estimates in TV-norm can be derived by using the Doeblin's minorisation condition, which always holds for irreducible Markov chains, and following the arguments in~\cite[Theorems 16.2.3, 16.2.4]{MeynTweedie09}; and estimates in Wasserstein distance can be derived under a geometric condition~\cite[Theorem 14.6]{LevinPeres17}. While all these approaches are used to compare the time-dependent solution to the steady state, the proof techniques generalise to the setting of this section, which compares two different time-dependent solutions. Nonetheless, we discuss our estimate since it uses relative entropy, which is not used in the literature, and the constants appearing in the estimates below are different from those appearing when using the techniques stated above.

Consider two time-dependent solutions $t\mapsto \zeta_t,\mu_t$ to the forward Kolmogorov equations~\eqref{eq:Nota-forKol} with two different initial conditions $\mu_0,\zeta_0\in\P(\Z)$. We recall from \eqref{ddt:var:H:zeta} that
\begin{align*}
    \frac{\dd}{\dd t} \Var_{\zeta_t}\biggl( \frac{\mu_t}{\zeta_t}\biggr) = -2  \calE_{\zeta_t}\biggl(\frac{\mu_t}{\zeta_t}\biggr), \qquad \frac{\dd}{\dd t} \RelEnt_{\zeta_t}\biggl( \frac{\mu_t}{\zeta_t}\biggr) = -  \RF_{\zeta_t}\biggl(\frac{\mu_t}{\zeta_t}\biggr).
\end{align*}
Therefore, using the generalised functional inequalities~\eqref{def:genPI},~\eqref{def:genLSI} we arrive at 
\begin{equation*}
    \begin{split}
        \frac{\dd}{\dd t} \Var_{\zeta_t}\biggl( \frac{\mu_t}{\zeta_t}\biggr) &= -2  \calE_{\zeta_t}\biggl(\frac{\mu_t}{\zeta_t}\biggr) 
        \leq - 2 \alpha_{\gPI}(\zeta_t) \Var_{\zeta_t}\biggl( \frac{\mu_t}{\zeta_t}\biggr), \\
        \frac{\dd}{\dd t} \RelEnt_{\zeta_t}\biggl( \frac{\mu_t}{\zeta_t}\biggr) &= -  \RF_{\zeta_t}\biggl(\frac{\mu_t}{\zeta_t}\biggr) \leq - \alpha_{\gLSI}(\zeta_t) \RelEnt_{\zeta_t}\biggl( \frac{\mu_t}{\zeta_t}\biggr).
    \end{split}
\end{equation*}
Note that this setup allows for considerably more flexibility compared to the classical setup since it allows us to prove concentration estimates for time-dependent solutions $\zeta_t$ rather than the steady state $\pi$ only. 

We now apply Gronwall's lemma to obtain
\begin{equation}\label{eq:gronwall-estimates} 
    \begin{split}
        \Var_{\zeta_t} \biggl( \frac{\mu_t}{\zeta_t}\biggr) &\leq \Var_{\zeta_0}\biggl( \frac{\mu_0}{\zeta_0}\biggr) \exp\left(-\int_0^t 2 \alpha_{\gPI}(\zeta_s) \dd s\right), \\
        \RelEnt_{\zeta_t} \biggl( \frac{\mu_t}{\zeta_t}\biggr) &\leq \RelEnt_{\zeta_0}\biggl( \frac{\mu_0}{\zeta_0}\biggr) \exp\left(-\int_0^t \alpha_{\gLSI}(\zeta_s) \dd s\right).
    \end{split}
\end{equation}
As the convergence rate is encoded in a time integral due to the time-dependence of $\zeta_t$, these estimate are not directly usable by themself. 
However, combining the lower bound on $\alpha_{\gLSI}$ in Theorem \ref{thm:pos}, together with $\alpha_{\gLSI} \leq 2\alpha_{\gPI}$ stated in  Proposition \ref{prop:LSI-implies-Poincare}, we find that 
\begin{equation}\label{eq:lower-bound-repetition}
    2\alpha_{\gPI}(\zeta_t)
    \geq \alpha_{\gLSI} (\zeta_t)
    \geq
    \dfrac{ (\zeta_t)_\ast M_* |\Z|^{-1} }{1 + \dfrac12 \log\dfrac{1}{(\zeta_t)_\ast}},
\end{equation}
where we recall that $(\zeta_t)_\ast$ is the minimal value of $\zeta_t$ and $M_*$ is the smallest positive entry of $M$. Using the lower bounds on solutions to forward Kolmogorov equations provided by \cite[Proposition~A.1]{HilderSharma24} we obtain the next  result which compares two solutions  to the forward Kolmogorov equation.

\begin{prop}\label{prop:long-term-est1}
    Let $M$ be an irreducible generator. 
    Then, for any $\delta \in (0,1)$ there exists an $\alpha_\ast>0$ such that for all initial data $\zeta_0, \mu_0 \in \P(\Z)$ and corresponding solutions $\zeta_t, \mu_t$ of the forward Kolmogorov equation \eqref{eq:full-forw-K} holds 
    \begin{equation}\label{eq:long-time-result-estimate1}
       0< \alpha_\ast \leq \alpha_{\gLSI}(\zeta_t) \leq 2\alpha_{\gPI} (\zeta_t)
    \end{equation}
    for all $t \geq \delta$ and 
    \begin{equation}\label{eq:error-estimates}
        \Var_{\zeta_t} \biggl( \frac{\mu_t}{\zeta_t}\biggr)\leq \Var_{\zeta_0}\biggl( \frac{\mu_0}{\zeta_0}\biggr) \min\left(1,e^{-\alpha_\ast (t - \delta)}\right),
        \ \  \RelEnt_{\zeta_t} \biggl( \frac{\mu_t}{\zeta_t}\biggr)\leq \RelEnt_{\zeta_0}\biggl( \frac{\mu_0}{\zeta_0}\biggr) \min\left(1,e^{-\alpha_\ast (t - \delta)}\right)
    \end{equation}
    for all $t \geq 0$.
\end{prop}
\begin{proof}
    We start with the lower bound on $\alpha_{\gLSI}(\zeta_t)$. For this, we note that since $\zeta_t$ converges to $\pi$ as $t \rightarrow \infty$ at an exponential rate, there exists a $\tau > 0$ such that $(\zeta_t)_\ast \geq \tfrac{\pi_\ast}{2} > 0$ for all $t > \tau$. In fact, such a $\tau$ can be chosen independent of $\zeta_0$. Next, we recall from \cite[Proposition~A.1]{HilderSharma24} that there exist constants $c(\delta) > 0$, $c(\delta,\tau) > 0$ and $N \in \N$ such that 
    \begin{equation*}
        \zeta_t \geq \begin{cases}
            c(\delta) t^N, & \text{ if } t \in [0,\delta), \\
            c(\delta,\tau), & \text{ if } t \in [\delta,\tau].
        \end{cases}
    \end{equation*}
    Together with \eqref{eq:lower-bound-repetition}, this shows that there exists an $\alpha_\ast > 0$ such that \eqref{eq:long-time-result-estimate1} holds. \comm{[$\alpha_*$ is actually indep of $\zeta_0$, but \cite{HilderSharma24} does not state it, and it is not so important, so lets avoid this discussion.]}

    For the second part of the proposition, we bound the exponentials in \eqref{eq:gronwall-estimates}. We note that it is sufficient to provide an estimate for the integral containing $\alpha_{\gLSI}$ since this implies the same estimate for $\alpha_{\gPI}$ due to Proposition \ref{prop:LSI-implies-Poincare}. First, we use that $\alpha_{\gLSI}(\zeta) \geq 0$ and thus,
    \begin{equation*}
        \RelEnt_{\zeta_t} \biggl( \frac{\mu_t}{\zeta_t}\biggr)\leq \RelEnt_{\zeta_0}\biggl( \frac{\mu_0}{\zeta_0}\biggr).
    \end{equation*}
    for all $t \geq 0$.
    For $t \geq \delta$, we can additionally estimate
    \begin{equation*} %\label{pfzj}
        \begin{split}
            \int_0^t \alpha_{\gLSI}(\zeta_s) \dd s &= \int_0^\delta \alpha_{\gLSI}(\zeta_s) \dd s + \int_\delta^t \alpha_{\gLSI}(\zeta_s) \dd s \geq \int_0^\delta \alpha_{\gLSI}(\zeta_s) \dd s + \alpha_\ast (t - \delta) \geq \alpha_\ast (t - \delta),
        \end{split}
    \end{equation*}
    where we again use that $\alpha_{\gLSI}(\zeta) \geq 0$. Plugging this into \eqref{eq:gronwall-estimates} then yields the estimate \eqref{eq:error-estimates}, which completes the proof.
\end{proof}

Recall $\alpha_{\PI}$ and $\alpha_{\LSI}$ as the classical Poincar\'e and log-Sobolev constants. Then, by exponential convergence to the stationary measure and continuity of the generalised Poincar\'e and log-Sobolev constants, see Proposition \ref{prop:gPI-char} and Theorem \ref{thm:continuity-alpha-LSI}, we find that $\alpha_{\gPI}(\zeta_t) \rightarrow \alpha_{\PI}$ and $\alpha_{\gLSI}(\zeta_t) \rightarrow \alpha_{\LSI}$ as $t \to \infty$. A direct adaptation of the proof of Proposition \ref{prop:long-term-est1} then yields the following corollary, which relates the long-term decay to the classical Poincar\'e and log-Sobolev constants.

\begin{cor}
    Let $M$, $\zeta_t$, $\mu_t$ be as in Proposition \ref{prop:long-term-est1}. Then, there exists a $\tau > 0$ such that 
    \begin{equation*} %\label{eq:long-term-estimate-3}
        \begin{split}
            \Var_{\zeta_t} \biggl( \frac{\mu_t}{\zeta_t}\biggr) &\leq \Var_{\zeta_0} \biggl( \frac{\mu_0}{\zeta_0}\biggr) e^{-\tfrac{\alpha_{\PI}}{2} t}, \qquad 
            \RelEnt_{\zeta_t} \biggl( \frac{\mu_t}{\zeta_t}\biggr) \leq \RelEnt_{\zeta_0} \biggl( \frac{\mu_0}{\zeta_0}\biggr) e^{-\tfrac{\alpha_{\LSI}}{2} t}
        \end{split}
    \end{equation*}
    for all $t \geq \tau$.
\end{cor}

\subsection{Clustering in Markov chains}\label{sec:application-CG}
 
We now turn to the second application of the generalised functional inequalities, which is clustering in Markov chains. Molecular systems and chemical kinetics are often modelled at a microscopic level by diffusions over large-dimensional complex energy landscapes~\cite{AllenTilsdesley17,Tuckerman23}. A typical situation is described in Figure~\ref{fig:energyLandscape}: the dynamics rapidly relaxes within the basin of attraction of a local minimum, while transition between neighbouring minima requires the crossing of energy barriers and therefore occurs on a longer timescales. Each such basin can be identified as a  \emph{micro-state}, and ignoring the fast equilibration within each basin, the slow dynamics can be approximated by a Markov chain on the set of micro-states with transition rates reflecting typical escape times between basins. This forms the basis of Markov state models typically used to study conformational dynamics~\cite{Prinz11}.

\begin{figure}[ht!]
\begin{center}
    \begin{tikzpicture}[scale=1.3]
    \draw [domain=-2.5:-0.5, samples=100] plot (\x, {cos(pi*\x r)*cos(pi*\x r)});
    \draw [domain=0.5:2.5, samples=100] plot (\x, {cos(pi*\x r)*cos(pi*\x r)});
    \draw [domain=-3:-2.5, samples=50] plot(\x, {3*cos(pi*\x r)*cos(pi*\x r)});
    \draw [domain=2.5:3, samples=50] plot(\x, {3*cos(pi*\x r)*cos(pi*\x r)});
    \draw [domain=-0.5:0.5, samples=50] plot(\x, {3*cos(pi*\x r)*cos(pi*\x r)});
    \draw (-2.7,-0.1) -- (-2.7,-0.3) (-2.7,-0.2)--(-.3,-0.2) node[midway,below]{macro-state} (-.3,-0.1) -- (-.3,-0.3);
    \draw (1.3,-0.1) -- (1.3,-0.3) (1.3,-0.2)-- (1.7,-0.2) (1.7,-0.1) -- (1.7,-0.3);
    \draw (1.5,-0.2) node[align=center,below=2pt]{micro-state};
    \end{tikzpicture}
    \caption{Energy landscape with two macro-states and six micro-states.}
    \label{fig:energyLandscape}
\end{center}
\end{figure}
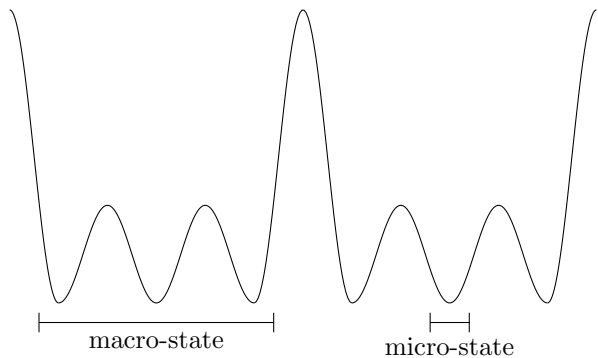

In some situations the energy landscape exhibits barriers of different heights, as in Figure~\ref{fig:energyLandscape}, wherein micro-states connected by smaller barriers equilibrate rapidly compared to transitions separated by large barriers. This leads to groups of micro-states, lumped together into so-called \emph{macro-states}, that exhibit quick internal equilibration within the group compared to slow transition across different macro-states. On sufficiently long time scales, the original Markov chain dynamics can be described by a reduced Markov chain on the macro-state space. This reduction is justified due to presence of pronounced disparities in the transition rates (or barrier heights) and underlies multiscale simulation and clustering/lumping approaches to Markov chains~\cite{CaoGillespiePetzold05,ELiuEijnden05,SchutteFischerHuisingaDeuflhard99}. Mathematically, these disparate transition rates correspond to the presence of explicit scale-separation in the system and considerable mathematical literature has been devoted to limit passage and averaging in multiscale jump processes, see~\cite{PavliotisStuart08,lahbabi13,zhang16,HilderPeletierSharmaTse20,MielkeStephan20,PeletierRenger20,landim2025} for a non-exhaustive list. 
While various reduction (or coarse-graining) approaches exist in the presence of pronounced rate disparities, coarse-graining to lower-dimensional Markov chains remains essential even when such disparities are not sharp, typically due to large system sizes. 
From this viewpoint, clustering even without explicit scale separation can be interpreted as constructing reduced Markov chains that approximate the slower components of the original dynamics.
These ideas have inspired various spectral and variational approaches to constructing clusters~\cite{DeuflhardWeber05,KubeWeber07}. While these techniques provide useful practical tools to cluster states together, they do not provide any error estimates to ascertain the quality of their approach. 

Inspired by coarse-graining for diffusion processes~\cite{legollLelievre10,Chorin2003}, the first and third author recently provided a systematic construction of a reduced Markov chain given any clustering. Furthermore, this work provides an error estimate comparing the reduced chain to the original chain under the assumption of a log-Sobolev inequality~\cite[Theorem 3.1]{HilderSharma24}. In this work, the LSI appeared as a technical assumption, and its interpretation was not transparent since the underlying dynamics was non-reversible and the reference measure which satisfied the LSI was not the steady state of the underlying generator~\cite[Remark 3.5 \& Section 6]{HilderSharma24}. 

With the generalised framework introduced in this paper, we can remove the log-Sobolev assumption entirely and obtain improved results both in the presence and in the absence of explicit scale separation in the system; see Proposition~\ref{prop:error-estimate-no-scale-sep} and Theorem~\ref{thm:error-estimate-ms} below. In the following, we consider coarse-grained Markov chains in two complementary regimes: Section~\ref{sec:estimates-without-scale-sep} treats coarse-grained Markov chains in the absence of explicit scale separation, illustrating how the framework applies to large systems where reduction is required for computational tractability rather than asymptotic reasons. In contrast, Section~\ref{sec:error-estimates-with-scale-sep} studies systems with explicit scale separation
where the reduced dynamics admit a clear multiscale interpretation. 

Moreover, in Section~\ref{sec:different-cg-maps}, we show how the generalised LSI and its associated lower bounds can be used to compare different clusterings, providing a quantitative notion of coarse-graining \emph{quality}. This was already suggested in \cite{HilderSharma24} in an example for a specific reversible Markov chain, which allows the LSI constant to be calculated. With the framework developed here, we can give a much more general approach for evaluating the quality of the coarse-graining map using the lower bound \eqref{eq:lowBound} and the upper bound \eqref{eq:upper-bound-gLSI}. As we demonstrate in Section \ref{sec:different-cg-maps} for a non-reversible example, this provides guidance for selecting clusterings that lead to accurate reduced Markov chains (see discussion in Section~\ref{sec:discussion}).
In addition, with a quality score it becomes possible to try to obtain a good coarse-graining map from data by machine learning, see Section~\ref{sec:discussion} for more details.

\subsubsection{Clustering in Markov chains without explicit scale separation}\label{sec:estimates-without-scale-sep}

We first discuss coarse-graining and error estimates for a general Markov chain without explicit scale separation. Before stating the result, we briefly recapitulate the setting in \cite{HilderSharma24} to fix notation. We consider a continuous-time Markov chain on a finite state space $\X$ with an irreducible generator $M \in \R^{|\X| \times |\X|}$. Then, the corresponding forward Kolmogorov equation for the evolution of the probability density reads 
\begin{equation}\label{eq:full-forw-K}
    \begin{dcases}
        \dfrac{\dd}{\dd t} \mu = M^T \mu, \\
        \mu_{t = 0} = \mu_0.
    \end{dcases}
\end{equation}
Our goal is to define an appropriate reduced Markov chain on a smaller state space $\Y$, which is encoded via a so-called \emph{coarse-graining map}
\[\xi : \X \to \Y.\]
We denote the level sets of $\xi$ by $\Lambda_y \coloneqq \{\xi^{-1}(x) \,:\, x \in \X\}$. The coarse-grained (or projected) dynamics on $\Y$ is \emph{exactly} characterised by the push-forward of $\mu$ under the map $\xi$, i.e.\ $t \mapsto \hat{\mu}_t\coloneqq \xi_{\#}\mu_t \in \P(\Y)$ defined as
\begin{equation*}
    \hat{\mu}_t(y) = \sum_{x \in \Lambda_y} \mu_t(x)
\end{equation*}
for any $y \in \Y$. 
The evolution of the exact coarse-grained dynamics $\hat{\mu}_t$ is explicitly described by 
\begin{equation}\label{eq:pushFor-Gen}
    \begin{dcases}
        \dfrac{\dd}{\dd t} \hat{\mu}_t = \hat{M}_t^T \hat{\mu}_t \\
        \hat{\mu}_{t=0}(y) = \sum_{x \in \Lambda_y} \mu_0(x),
    \end{dcases} \quad \text{ with } \quad \hat{M}_t(y_1,y_2) \coloneqq \sum_{x_1 \in \Lambda_{y_1}, x_2 \in \Lambda_{y_2}} M(x_1,x_2) \mu_t(x_1 | y_1),
\end{equation}
see \cite[Lemma 2.4]{HilderSharma24} for a derivation. While the evolution of  $\hat{\mu}_t$ is the exact dynamics induced by the full dynamics $\mu_t$ under the coarse-graining map $\xi$, it is impractical in applications since the generator is time-dependent and, in fact, requires full knowledge of the full dynamics $\mu_t$. 

Instead, motivated by recent developments in computational statistical mechanics~\cite{legollLelievre10,Chorin2003}, an \emph{approximate effective} dynamics $t \mapsto \eta_t \in \P(\Y)$ is introduced in~\cite{HilderSharma24} which evolves according to
\begin{equation}\label{eq:eff-dyn-non-cg}
    \dfrac{\dd}{\dd t} \eff_t = N^T\eff_t, \ \ \text{ with } \ \  N(y_1,y_2)\coloneqq \sum_{x_1\in\Lambda_{y_1},x_2\in\Lambda_{y_2} } M(x_1,x_2)\stat(x_1|y_1),
\end{equation}
where $\rho(\cdot | y) \in \P(\Lambda_y)$ for $y \in \Y$ is the family of conditional measures corresponding to the full steady state $\rho \in \P(\X)$ of $M$, i.e.\  
\begin{align*}
\forall x\in \Lambda_y: \ \rho^\e(x|y)=\frac{\rho^\e(x)}{\sum_{x'\in \Lambda_y}\rho^\e(x') }. 
\end{align*}
To construct the effective dynamics, one only needs to know the full-space steady state $\rho$ rather than the full-space dynamics. This generator $N$ of the effective dynamics~\eqref{eq:eff-dyn-non-cg} mimics the coarse-grained generator $\hat M_t$~\eqref{eq:pushFor-Gen} with the crucial difference that the cluster dynamics (i.e.\ dynamics on the level-sets of $\xi$) has been equilibrated -- we refer interested readers to~\cite{HilderSharma24} for details. 

The key question in the investigation in \cite{HilderSharma24} is under what conditions this effective dynamics is a good approximation of the coarse-grained dynamics, and by extension, the full dynamics. For this setting, the first and third author provided an error estimate, which measures the error between the coarse-grained dynamics $\hat{\mu}_t$ and the effective dynamics $\eta_t$ in relative entropy, that is, it provides a bound on $\RelEnt_{\eta_t}(\hat{\mu}_t / \eta_t)$, see \cite[Theorem 3.1]{HilderSharma24} and \eqref{eq:error-est-non-cg}. To formulate this estimate, we define for any $y \in \Y$ the generator $M^y \in \R^{|\Y| \times |\Y|}$ as the restriction of $M$ to $\Lambda_y \times \Lambda_y$ with an appropriately modified diagonal such that $M^y$ is again a generator. The key assumption in \cite[Theorem 3.1]{HilderSharma24} is the existence of $\alpha > 0$ such that $\alpha \le \alpha_{\gLSI}(\rho(\cdot\vert y), M^y)$ for all $y \in \Y$. To provide such an $\alpha$, we introduce the symmetrised generators \comm{[complicated notation st we are aligned with the Mathematica code]}
\begin{equation}\label{eq:SymGen-Comp}
    M^{y,\rho(\cdot\vert y)} \coloneqq \dfrac{1}{2} (M + D_{\rho(\cdot \vert y)} M^T D_{\rho(\cdot \vert y)}^{-1}).
\end{equation}
Here we use the notation $D_{\rho(\cdot \vert y)} = \mathrm{diag}(\rho(\cdot | y))$, i.e., the diagonal matrix with the entries of $\rho(\cdot | y)$ on its diagonal. Note that the second term in the right-hand side of~\eqref{eq:SymGen-Comp} is the explicit form of the adjoint to $M$ in $L^2(\zeta)$, which is precisely the second term  appearing in the symmetrised generator $M^\zeta$ \eqref{eq:symGen}.
Then, applying Theorem \ref{thm:pos}, we find that
\begin{equation}\label{eq:alpha}
    \alpha \coloneqq \min_{y \in \Y} \dfrac{\lambda_y}{2 + \log\dfrac{1}{\rho(\cdot | y)_*}}
\end{equation}
with $\lambda_y$ the spectral gap of $M^{y,\rho(\cdot\vert y)}$
and $\rho(\cdot | y)_*$ the minimum of $\rho(\cdot | y)$, satisfies $\alpha \le \alpha_{\gLSI}(\rho(\cdot\vert y), M^y)$ for all $y \in \Y$. Recall that Theorem \ref{thm:pos} also provides a more explicit lower bound which does not require the calculation of the spectral gap. Although both lower bounds in Theorem \ref{thm:pos} yield strict positivity of the $\alpha_{\gLSI}(\rho(\cdot\vert y), M^y)$ for all $y \in \cY$, we use the sharper lower bound given by $\alpha$ here. A direct application of \cite[Theorem 3.1, Remark 3.2]{HilderSharma24} then yields the following result. 

\begin{prop}\label{prop:error-estimate-no-scale-sep}
    Let $\mu \in C^1([0,\infty);\P(\X))$ be a solution to the full forward Kolmogorov equation \eqref{eq:full-forw-K} with corresponding coarse-grained dynamics $\hat{\mu}$ and let $\eta \in C^1([0,\infty); \P(\Y))$ be the solution to the effective dynamics \eqref{eq:eff-dyn-non-cg} with initial condition $\hat{\mu}_0,\eff_0 \in \P_+(\Y)$ satisfying $\hat{\mu}_0(y), \eta_0(y) \geq c_0$ for some $c_0 > 0$ and all $y \in \Y$. Then, for any fixed $T > 0$ there exists a constant $C = C(L,|\X|,T,c_0)$ such that
    \begin{equation}\label{eq:error-est-non-cg}
        \begin{split}
            \RelEnt_{\eta_t}\left(\dfrac{\hat{\mu}_t}{\eta_t}\right) &\leq 2\RelEnt_{\eta_0}\left(\dfrac{\hat{\mu}_0}{\eta_0}\right) + \dfrac{C}{\min_y(\alpha_{\gLSI}(\rho(\cdot |y), M^y))} \left[\RelEnt_\rho\left(\dfrac{\mu_0}{\rho}\right) - \RelEnt_\rho\left(\dfrac{\mu_t}{\rho}\right)\right]
        \end{split}
    \end{equation}
    for all $t \in [0,T]$. In particular, the implicit prefactor $\min_y(\alpha_{\gLSI}(\rho(\cdot |y), M^y))$ can be replaced by the explicit constant $\alpha > 0$ given in \eqref{eq:alpha}.
\end{prop}

\begin{rem}
    As pointed out in \cite[Remark 3.2]{HilderSharma24}, it is unclear if the estimate \eqref{eq:error-est-non-cg} can be proven for $T = \infty$. However, one can obtain the weaker estimate
    \begin{equation*}
        \RelEnt_{\eta_t}\left(\dfrac{\hat{\mu}_t}{\eta_t}\right) \leq \RelEnt_{\eta_0}\left(\dfrac{\hat{\mu}_0}{\eta_0}\right) + \dfrac{C}{\sqrt{\alpha}} \left[\RelEnt_\rho\left(\dfrac{\mu_0}{\rho}\right) - \RelEnt_\rho\left(\dfrac{\mu_t}{\rho}\right)\right]^{\tfrac{1}{2}}
    \end{equation*}
    for all $t > 0$ with a constant $C$ now depending on $L, N$ and $\rho$, see \cite[Theorem 3.1]{HilderSharma24}. In fact, this estimate even holds without assuming strictly positive initial data.
\end{rem}

We should point out that it is possible to construct another natural reduced dynamics on $\Y$ inspired by multiscale problems in Markov chains (see Section~\ref{sec:error-estimates-with-scale-sep} for this connection). The following remark briefly discusses this reduced dynamics in the the absence of scale separation.
\begin{rem}\label{rem:Aver-NoScale}
    The restriction $M^y\in\R^{|\Y|\times|\Y|}$ of the generator $M$ to $\Lambda_y\times\Lambda_y$ is assumed to be a generator in the setting of this section, which can always be guaranteed by modifying the diagonal elements of the restriction. If we additionally assume that $M^y$ is irreducible for every $y\in\Y$, then there exists $\rho_y\in\P(\Y)$ such that $(M^y)^T\rho_y=0$, i.e., $\rho_y$ is the steady state of the restricted generator $M^y$. Note that, in the non-reversible case, the conditional steady state $\rho(\cdot|y)$ is generally not the same as the steady state $\rho_y$ within each cluster (see Section~\ref{sec:different-cg-maps} for one such explicit example). Following the same strategy as in the construction of effective dynamics above, we can introduce the reduced dynamics 
    \begin{equation*}
    \dfrac{\dd}{\dd t} \aver_t = (M^{\av})^T\aver_t, \ \ \text{ with } \ \  M^{\av}(y_1,y_2)\coloneqq \sum_{x_1\in\Lambda_{y_1},x_2\in\Lambda_{y_2} } M(x_1,x_2)\rho_{y_1}(x_1).
\end{equation*}
    Here the $\aver_t$ refers to the \emph{averaged dynamics} in multiscale problems, see~\eqref{eq:averaged_dynamics} below. While this dynamics is traditionally used only in the context of multiscale problems, it remains well-defined even in the absence of explicit scale separation, as is the case here. However, our proof techniques from Proposition~\ref{prop:error-estimate-no-scale-sep} do not carry over to this averaged reduced dynamics. Specifically, the proof of \cite[Lemma 3.8]{HilderSharma24} fails since the push-forward steady state $\xi_\# \rho$ is generally not the steady state of $M^{\av}$. This points to the main issue that, in general, there is no reason why the coarse-grained generator $\hat M_t$~\eqref{eq:pushFor-Gen} should be close to $M^{\av}$. This is in stark contrast to the effective generator $N$, see \eqref{eq:eff-dyn-non-cg}, which is the limit of $\hat{M}_t$ for $t \to \infty$. A notable exception to this issue are reversible Markov chains where $\rho_y(\cdot)=\rho(\cdot|y)$, see \cite[Lemma 4.9]{HilderSharma24}, 
    and therefore the effective and averaged dynamics coincide.  
\end{rem}

\subsubsection{Clustering in Markov chains with explicit scale separation}\label{sec:error-estimates-with-scale-sep}

The estimate \eqref{eq:error-est-non-cg} is useful in particular if $\alpha$ is large. 
To illustrate this, we now consider the error estimates for a system with explicit scale separation, which is reflected in a small model parameter $\eps$. Specifically, we consider the behaviour of a particle moving between the micro-states within an energy landscape with energy barriers of vastly different sizes, similar to Figure~\ref{fig:energyLandscape}. We model this as a Markov jump process on a discrete state space $\X = \Y\times \Z$, where $\Y$ labels the macro-state and $\Z$ labels the micro-state within a particular macrostate. We will make two assumptions throughout:
\begin{itemize}
    \item there are only two macro-states i.e.\ $\Y = \{0,1\}$;
    \item all macro-state contains equal number of micro-states i.e.\ $\Z = \{0, \dots, n-1\}$ with $n \geq 1$. 
\end{itemize}
This is reflected in Figure~\ref{fig:energyLandscape}, where the state space (which labels each basin) is given by $\X=\Y\times\Z\coloneqq \{0,1\}\times \{0,1,2\}$. 
Both these assumptions have been made purely for the sake of notational simplicity and all the following results can straightforwardly be extended to multiple macro-states containing varying number of micro-states. We refer to \cite[Section 6]{HilderSharma24} for a more detailed discussion.

To make these ideas concrete, consider a family of Markov chains parametrised by $\e>0$ with corresponding forward Kolmogorov equations 
\begin{equation}\label{eq:eps-KolEq}
\begin{dcases}
\dfrac{\dd}{\dd t} \mu^\e = (L^\e)^T \mu^\e, \\
\mu^\e_{t = 0} = \mu_0^\eps,
\end{dcases}
\end{equation}
on $\X = \Y \times \Z$, where $ \mu^\e_t$ is the probability distributions of the Markov chain with the initial datum $\mu_0^\eps$. The generator $L^\e$ is assumed to admit the form 
\[
L^\e= \frac{1}{\e}Q + \sG,
\]
where $Q, G$ are $\e$-independent generators, and $\varepsilon > 0$ is a small parameter that represents the ratio of the two time scales.
In the language of Markov chains, $Q$ describes the $\cO(\frac1\e)$ fast dynamics on the micro-states $\Z$ belonging to a single macro-state $y \in \Y$, and $G$ describes the $\cO(1)$ slow dynamics of jumps between different macro-states $y,y' \in \Y$. 
More precisely,
we further decompose $Q, G$ into the matrices $Q_y,D_y, G_{y,1-y}\in\R^{n\times n}$ for $y\in\Y$ as
\begin{equation}\label{eq:CG-res-gen}
L^\e= \frac{1}{\e}Q + \sG \coloneqq \frac{1}{\e}\begin{pmatrix}Q_0 & 0 \\ 0 & Q_1\end{pmatrix} +  \begin{pmatrix}D_0 & \sG_{0,1} \\
\sG_{1,0} & D_1 \end{pmatrix},
\end{equation}
i.e.\ with (writing $x=(y,z) \in \cX$)
\[
Q((y,z),(y',z')) = \begin{cases}
Q_y(z,z') &\text{if\, $y'=y$}\\
0 &\text{otherwise}
\end{cases},\qquad \sG((y,z),(y',z')) = \begin{cases}
D_y(z,z') & \text{if $y' = y$} \\ 
\sG_{y,y'}(z,z') &\text{otherwise}
\end{cases}
\]
and $D_y$ diagonal matrices.
The matrix $Q_y\in\R^{n\times n}$ encodes the fast jumps between micro-states within the $y$-th macro-state. The matrix $G_{y,1-y}\in\R^{n\times n}$ encodes the slow transition from the $y$-th macro-state to the $(1-y)$-th macro-state. Finally, $D_y$ ensures that $G$ is a generator. 

Since $L^\e, Q, G$ are generators and $Q, G$ are independent of $\e$, they satisfy
\begin{align*} %\label{CG-L-ass} 
\forall x\in\X &: \ \sum_{x'\in\X}Q(x,x') =0=\sum_{x'\in\X} \sG(x,x'),
\\ %\label{eq-def:D}
\forall z\in\Z &: \ D_y(z) = -\sum_{z'\in\Z} \sG_{y,1-y}(z,z').
\end{align*}
We assume that $L^\e$ is irreducible, and therefore~\eqref{eq:eps-KolEq} admits a stationary solution $\stat^\e\in\P_+(\X)$. Additionally, we assume that $Q_0$ and $Q_1$ are irreducible generators as well. Consequently, the dynamics driven by $Q_y$ on $\Z$ for $y=0,1$ admit stationary measures $\stat_y \in\P_+(\Z)$.

This setting is of coarse-graining type in the sense 
that for $0<\e\ll 1$ the dynamics within the macro-state equilibrates, which allows the reduction to a jump process on $\Y$ in the limit $\eps \rightarrow 0$, i.e.\ we expect that only relevant dynamics in the limit is between macro-states.

To make this precise, we introduce a coarse-graining map as the projection onto the slow variables encoded in $\Y$
\begin{equation}\label{eq:CGDyn-def}
\xi:\X\rightarrow\Y \ \ \text{with} \ \ \xi (y,z) = y. 
\end{equation} 
As in the case without explicit scale separation, we describe the slow coarse-grained dynamics by $t\mapsto \hat\mu^\e_t \coloneqq \xi_\#\mu^\e_t \in \P(\Y)$ defined as
\begin{equation*}
\forall y\in \Y: \ \  \hat\mu^\e_t(y) \coloneqq \sum_{z\in \Z}\mu^\e_t((y,z)).  
\end{equation*}
We use $\Lambda_y \coloneqq \xi^{-1}(y) \subset \X$ to describe the the `copy' of $\Z$ that belongs to a macro-state $y \in \Y$. 

Deriving the limiting dynamics as $\eps\to 0$ has received considerable attention in the literature. Under some mild assumptions, the solution to the coarse-grained dynamics converges to the averaged dynamics, i.e.\ $\cg^\e\to\aver$, where $t\mapsto \aver_t\in \P(\Y)$ solves
\begin{equation}\label{eq:averaged_dynamics}
    \dfrac{\dd}{\dd t} \aver_t = \bigl( L^{\av}\bigr)^T\aver_t,
\end{equation}
see~\cite[Chapter 16]{PavliotisStuart08} for convergence of backward equations,~\cite[Theorem 1]{lahbabiLegoll13} for a martingale approach and~\cite[Section 3]{HilderPeletierSharmaTse20} for a variational approach.
Here the limiting generator $L^{\av}\in \mathbb R^{|\Y|\times |\Y|} = \R^{2 \times 2}$ is defined as  
\begin{equation}\label{eq:aver-gen}
L^{\av}\coloneqq \begin{pmatrix}
-\lambda_0 & \lambda_0  \\ \lambda_1 & -\lambda_1
\end{pmatrix}, \ \ 
\lambda_y \coloneqq \sum_{z,z'\in\Z} \stat_y(z) \sG_{y,1-y}(z,z'),
\end{equation}
where $\stat_y$ is the stationary measure corresponding to $Q_y$. Note that this coincides with the reduced dynamics proposed in Remark~\ref{rem:Aver-NoScale} in the absence of scale separation.

While these results describe the case of `infinite scale separation', as discussed in the case without explicit scale separation, it is usually more relevant in applications to obtain a reduced model in the case of small, but finite, $\eps > 0$, or even for settings without explicit scale separation. Motivated the previous Section \ref{sec:estimates-without-scale-sep}, we again define the effective dynamics $t\mapsto \eff_t^\e \in\P(\Y)$ as
\begin{equation}\label{eq:eff}
\dfrac{\dd}{\dd t} \eff_t^\e = \bigl( N^\e \bigr)^T\eff_t^\e, \ \ \text{ with } \ \  N^\e(y_1,y_2)\coloneqq \sum_{x_1\in\Lambda_{y_1},x_2\in\Lambda_{y_2} } L^\e(x_1,x_2)\stat^\e(x_1|y_1),
\end{equation}
where $\stat^\e(\cdot|y)\in\P(\Lambda_y)$ are the conditional measures corresponding to the full steady state $\stat^\e$. We point out that $N^\eps$ is an irreducible generator with steady state $\hat{\rho}^\eps$ given by
\begin{equation*}
    \hat{\rho}^\eps(y) \coloneqq \sum_{x \in \Lambda_y} \rho^\eps(x | y)
\end{equation*}
for all $y \in \Y$, see~\cite[Lemma 2.5, Proposition 2.6]{HilderSharma24}. 

We note that the effective dynamics converges to the averaged dynamics, i.e.\ $\eff^\e\to \aver$ as $\e\to 0$~\cite[Theorem 4.4]{HilderSharma24}. More interestingly, the effective dynamics is indeed a good approximation of the true slow-variable dynamics~\cite[Theorem 4.5]{HilderSharma24}, i.e.\ under suitable assumptions we find
\begin{align*} %\label{eq:error-estiamte-HS22}
 \sup_{t \in [0,T]} \RelEnt_{\eta_t^\eps}\left(\dfrac{\cg_t^\e}{\eta_t^\eps}\right) \leq C \eps.
\end{align*}
The main assumption is the existence of a uniform log-Sobolev inequality, that is, there exists an $\alpha > 0$ such that $\alpha < \alpha_{\gLSI}(\rho^\eps(\cdot \vert y), Q_y)$ for all $\eps \in (0,1)$ and each $y \in \Y$. This is similar to the case without explicit scale separation, but we additionally require uniformity in $\eps \in (0,1)$.
Note that similar assumptions also appear in the coarse-graining literature on diffusion processes~\cite{legollLelievre10, ZhangHartmannSchutte16, DLPSS18, HartmannNeureitherSharma20} for the log-Sobolev inequality and~\cite{LegollLelievreOlla17,LegollLelievreSharma18,LelievreZhang18} for the Poincar\'e inequality. 
Using the framework introduced in this paper, we can remove this assumption entirely, as we will show in the remainder of this section.

We stress that for non-reversible processes $L^\eps$, the conditional steady-state $\rho^\eps(\cdot|y)$ is in general not the steady state of the generator $Q_y$ on the level-set $\Lambda_y$ since $\rho^\eps(\cdot|y)$ generically depends on $\eps$ while $Q_y$ does not, see Section \ref{sec:different-cg-maps} for an explicit example. However, if $L^\eps$ is reversible, its stationary measure $\rho^\eps$ is $\eps$-independent, see~\cite[Lem.~4.9]{HilderSharma24}, and $\rho^\eps(\cdot|y)$ is the steady state of $Q_y$.  

The existence and interpretation of such an $\alpha > 0$ 
was left as an open question in \cite{HilderSharma24}. Now, with the non-equilibrium functional inequalities developed in this paper, we can answer it.
For this, we now apply the general set-up to the scale-separated Markov chain with generator $L^\eps$, see \eqref{eq:CG-res-gen}, and obtain the result, which removes the assumption of a uniform log-Sobolev constant in \cite[Thm.~4.5]{HilderSharma24}.

\begin{theorem}\label{thm:error-estimate-ms}
    Let $Q, G$ be as defined in \eqref{eq:CG-res-gen}. Then, there exist constants $\eps_0, c_0, C_0, C_1,C_2, \beta > 0$ such that the following result holds. For $\eps \in (0,\e_0)$ let $t \mapsto \mu^\eps_t \in \P(\X)$ and $t \mapsto \eta^\eps_t\in \P(\Y)$ be the solutions to
    \begin{equation*}
        \begin{dcases}
            \dfrac{\dd}{\dd t} \mu^\eps_t = (L^\eps)^T \mu_t^\eps, \\
            \mu_{t = 0}^\eps = \mu_0^\eps,
        \end{dcases} \qquad \begin{dcases}
            \dfrac{\dd}{\dd t} \eta^\eps_t = (N^\eps)^T \eta_t^\eps, \\
            \eta^\eps_{t = 0} = \eta_0^\eps,
        \end{dcases}
    \end{equation*}
    where we assume the following on the initial conditions $\mu_0^\eps, \eta_0^\eps$: 
    \begin{itemize}
        \item (strictly positive initial data) $\hat{\mu}_0^\eps(x) \geq c_0$ and $\eta_0^\eps(y) \geq c_0$ for all $\eps \in (0, \eps_0)$ and for all $x \in \mathcal{X}$ and all $y \in \mathcal{Y}$ respectively,
        \item (convergence) $\mu^\eps_0$, $\hat{\mu}_0^\eps = \xi_\#\mu^\eps_0$ (see \eqref{eq:CGDyn-def}) and $\eta_0^\eps$ converge in $\P(\mathcal{Y})$ as $\eps \rightarrow 0$, 
        and
        \item (relative entropy estimate) for all $\eps \in (0,\eps_0)$
        \begin{equation*} %\label{eq:ms-convergence-init}
        \RelEnt_{\eta_0^\eps}\left(\frac{\hat{\mu}_0^\eps}{\eta_0^\eps}\right) \leq C_0 \eps.
    \end{equation*}
    \end{itemize}
    Then, the estimate
    \begin{equation}\label{eq:better-error-estimate-ms}
        \RelEnt_{\eta_t^\eps} \left(\frac{\hat{\mu}_t^\eps}{\eta_t^\eps}\right) \leq \eps C_1 e^{-\beta t}
    \end{equation}
    holds for all $t \geq 0$ and all $\eps \in (0,\eps_0)$. 
    In particular, the time-uniform estimate
    \begin{equation}\label{eq:better-error-estimate-ms-uniform}
        \sup_{t \geq 0} \RelEnt_{\eta_t^\eps} \left(\frac{\hat{\mu}_t^\eps}{\eta_t^\eps}\right) \leq C_2 \eps
    \end{equation}
    holds for all $\eps \in (0,\eps_0)$.
\end{theorem}

\begin{rem}
    We briefly comment on the first two assumptions in Theorem \ref{thm:error-estimate-ms}. The first assumption is consistent with the assumptions in Proposition \ref{prop:error-estimate-no-scale-sep}, and we add the uniformity in $\eps > 0$ to avoid degeneracy as $\eps \to 0$. The second assumption that $\hat{\mu}_0^\eps$ and $\eta_0^\eps$ converge as $\eps \to 0$ is used in the proof of \cite[Theorem 4.5]{HilderSharma24} to obtain that both the coarse-grained dynamics $\hat{\mu}_t^\eps$ and the effective dynamics $\eta_t^\eps$ converge to the $\eps$-independent averaged dynamics $\mu_t^{\av}$ given by \eqref{eq:averaged_dynamics} uniformly on a fixed, finite time interval $[0,T]$. This is then used to obtain an $\eps$-independent lower bound on $\hat{\mu}_t^\eps$ and $\eta_t^\eps$ on $[0,T]$. Additionally, we assume the convergence of the full initial data $\mu_0^\eps$, which is used in the proof of \cite[Theorem 4.5]{HilderSharma24} to guarantee convergence of the relative entropy of $\mu_0^\eps$ and $\rho^\eps$ as $\eps \to 0$. Since this is only needed to obtain an $\eps$-uniform bound on the relative entropy of $\mu_0^\eps$ and $\rho^\eps$, we conjecture that it can be removed. However, we still state it here for convenience.
\end{rem}

\begin{proof}
    We first prove that for all $T > 0$ 
    there exists a constant $C(T) > 0$ such that
    \begin{equation}\label{eq:finite-time-estimate}
        \sup_{t \in [0,T]} \RelEnt_{\eta_t^\eps} \left(\frac{\hat{\mu}_t^\eps}{\eta_t^\eps}\right) \leq C(T) \eps
    \end{equation}
    for all $\eps \in (0,\eps_0)$. This follows from an application of \cite[Theorem 4.5, Equation (50)]{HilderSharma24}. 
    For this, the only part left to check is  \cite[Thm.~4.5 (A2)]{HilderSharma24}, i.e.\  there exist $\eps_0 > 0$ and $\alpha > 0$ such that for all $\eps \in (0, \eps_0)$, $y \in \mathcal{Y}$ and $\nu \in \P(\mathcal{Z})$ it holds that
    \begin{equation*}
        \RelEnt_{\rho^\eps(\cdot | y)}\left(\frac{\nu}{\rho^\eps(\cdot | y)}\right) \leq \frac{1}{\alpha} \RF_{\rho^\eps(\cdot | y)}\left(\frac{\nu}{\rho^\eps(\cdot | y)},Q_y\right),
    \end{equation*}
    where $\RF_{\rho^\eps(\cdot | y)}$ is the generalised Fisher information with respect to the generator $Q_y$, see \eqref{eq:CG-res-gen}, and $\rho^\eps(\cdot | y)$ is the conditional measure to the full steady state $\rho^\eps$ of $L^\eps$. Thus, we need to show that $\alpha_{\gLSI}(\rho^\eps(\cdot | y), Q_y) \geq \alpha$ 
    for some $\alpha > 0$ independent of $\eps \in (0,\eps_0)$. To see this, we note that $\rho^\eps(\cdot | y) \rightarrow \pi_y$ as $\eps \rightarrow 0$, where $\pi_y$ is the positive stationary measure of the generator $Q_y$, see \cite[Lemma 3.3]{HilderPeletierSharmaTse20}.
    Since $ \alpha_{\gLSI}(\pi_y, Q_y) = \alpha_{\LSI}(Q_y) > 0$ for each $y \in \Y$, we have $\tilde{\alpha} \coloneqq \min_{y \in \Y} \alpha_{\gLSI}(\pi_y, Q_y)  > 0$. Then, the continuity of $\nu \mapsto \alpha_{\gLSI}(\nu, Q_y)$ for each $y$ shows that there exists an $\eps_0 > 0$ such that $\min_{y \in \Y} \alpha_{\gLSI}(\rho^\eps(\cdot | y), Q_y) \geq \tfrac{\tilde{\alpha}}{2} =: \alpha$ for all $\eps \in (0,\eps_0)$. Note that $\eps_0$ can be chosen uniformly in $y$ since $\mathcal{Y}$ is finite. Then, all assumptions of \cite[Theorem~4.5]{HilderSharma24} are satisfied, which yields \eqref{eq:finite-time-estimate}.

    Next, we establish a long-time estimate. Although this follows from similar arguments as in Section \ref{sec:application-long-term-estimates}, we give a detailed proof since we have to carefully track the $\eps$ dependence. We recall that
    \begin{equation}\label{eq:time-derivative-for-error-estimate}
        \dfrac{\dd}{\dd t} \RelEnt_{\eta_t^\eps} \left(\frac{\hat{\mu}_t^\eps}{\eta_t^\eps}\right) = - \RF_{\eta_t^\eps}\left(\frac{\hat{\mu}_t^\eps}{\eta_t^\eps}\right)  \leq - \alpha_{\gLSI}(\eta_t^\eps, N^\e) \RelEnt_{\eta_t^\eps} \left(\frac{\hat{\mu}_t^\eps}{\eta_t^\eps}\right).
    \end{equation}
    We now show that there exist $T_0 > 0$ and $\eps_0 > 0$ such that
    \begin{equation}\label{eq:cg-lower-bound-gLSI}
        \alpha_{\gLSI}(\eta_t^\eps, N^\eps) \geq \dfrac{1}{2}\alpha_{\gLSI}(\pi^\mathrm{av},L^\mathrm{av}) =: \beta > 0
    \end{equation}
    for all $t \geq T_0$ and all $\eps \in (0,\eps_0)$. The proof relies on the continuity of the map $(\zeta,M) \mapsto \alpha_{\gLSI}(\zeta,M)$ at $(\zeta,M) = (\pi^\mathrm{av},L^\mathrm{av})$, see Theorem \ref{thm:continuity-alpha-LSI} and Remark \ref{rem:continuity-gLSI-with-generator}, and the observation that the convergence $(\eta_t^\eps, N^\eps) \to (\pi^\mathrm{av},L^\mathrm{av})$ as $\eps \to 0$ and $t \to \infty$ is uniform with respect to the initial condition $\eta_0^\e$.
    To prove the latter, it is sufficient to show that for all $\delta > 0$ there exists $T_1>0$ large enough and $\eps_0>0$ small enough such that the estimate
    \begin{equation*}
        \|\eta_t^\eps - \pi^\mathrm{av}\|_{\TV} + \|N^\eps - L^\mathrm{av}\| < \delta
    \end{equation*}
    holds for any $t \geq T_1$ and any $\eps \in (0,\eps_0)$. First, we recall from \cite[Equation~(43)]{HilderSharma24} that
    \begin{equation}\label{eq:quant-estimate-diff-generators}
        \|N^\eps - L^\mathrm{av}\| \leq C \eps
    \end{equation}
    for all $\eps \in (0, \eps_0)$. Note that the argument below does not need the linear decay in $\eps$ and any decay would suffice. 
Second, we estimate
    \begin{equation*}
        \|\eta_t^\eps - \pi^\mathrm{av}\|_{\TV} \leq \|\eta_t^\eps - \hat{\rho}^\eps\|_{\TV} + \|\hat{\rho}^\eps - \pi^\mathrm{av}\|_{\TV},
    \end{equation*}
    where  $\hat{\rho}^\eps \in \P_+(\Y)$ is the steady state of $N^\eps$. 
    The first term is expected to be small for sufficiently large $t$ due to the convergence of $\eta_t^\eps$ to the steady state $\hat{\rho}^\eps$ and the second term to be small for sufficiently small $\eps > 0$ due to $\hat{\rho}^\eps \to \pi^\mathrm{av}$ as $\eps \to 0$, see \cite[Lemma~4.3]{HilderSharma24} for a qualitative proof of the latter. We now quantify these convergences.

    For the second term, we note that 
    \begin{equation*}
        \|\hat{\rho}^\eps - \pi^\mathrm{av}\|_{\mathrm{TV}} \leq C\eps
    \end{equation*}
    for all $\eps \in (0,\eps_0)$. Since this result is potentially of independent interest, we prove this estimate in the following separate Lemma \ref{lem:error-steady-states}.
    For the first term, by \cite[Proposition~A.2]{HilderSharma24}, there exist constants $C(\eta_0^\eps), D(N^\eps) > 0$ such that
    \begin{equation*}
        \|\eta_t^\eps - \hat{\rho}^\eps\|_{\TV} \leq C(\eta_0^\eps) e^{-D(N^\eps) t}
    \end{equation*} 
    for all $t \geq 0$. Next, we recall that $N^\eps \rightarrow L^{\av}$ as $\e \to 0$, see \eqref{eq:quant-estimate-diff-generators}. Revisiting the proof of \cite[Proposition~A.2]{HilderSharma24} we can, in fact, find constants $C, D > 0$ independent of $\e, t$ and the initial data $\eta_0^\eps$ such that 
    \begin{equation*}
        \|\eta_t^\eps - \hat{\rho}^\eps\|_{\TV} \leq C e^{-D t}
    \end{equation*}
    for all $t \geq 0$ and $\eps \in (0,\eps_0)$ (potentially choosing a smaller $\eps_0$). 
    Here, we use that $D(N^\eps)$ can be bounded from below by the spectral gap of $N^\eps$. To obtain the independence of the initial data, we note in the proof of \cite[Proposition~A.2]{HilderSharma24} that $\P(\Y)$ can be interpreted as a bounded subset of $\R^{|\Y|}$. Since the spectrum of a finite-dimensional matrix is continuous in the matrix entries, the spectral gap of $N^\eps$ converges to the (positive) spectral gap of $L^{\av}$. In summary, we obtain that
    \begin{equation*}
        \|\eta_t^\eps - \pi^\mathrm{av}\|_{\TV} \leq C (e^{-D t} + \eps),
    \end{equation*}
    which can be made arbitrarily small uniformly in $t \geq T_1$ and $\eps \in (0,\eps_0)$ by choosing $T_1$ sufficiently large and $\eps_0 > 0$ sufficiently small.

    Using \eqref{eq:time-derivative-for-error-estimate} and \eqref{eq:cg-lower-bound-gLSI} we thus obtain that
    \begin{equation*}
        \dfrac{\dd}{\dd t} \RelEnt_{\eta_t^\eps} \left(\frac{\hat{\mu}_t^\eps}{\eta_t^\eps}\right) \leq -\beta \RelEnt_{\eta_t^\eps} \left(\frac{\hat{\mu}_t^\eps}{\eta_t^\eps}\right)
    \end{equation*}
    for all $t \geq T_1$ and $\eps \in (0,\eps_0)$.
    An application of Gronwall's inequality thus yields 
    \begin{equation}\label{eq:large-time-estimate}
        \RelEnt_{\eta_t^\eps} \left(\frac{\hat{\mu}_t^\eps}{\eta_t^\eps}\right) 
        \leq \RelEnt_{\eta_{T_1}^\eps} \left(\frac{\hat{\mu}_{T_1}^\eps}{\eta_{T_1}^\eps}\right) e^{-\beta (t - T_1)} 
        \leq \e C(T_1) e^{\beta T_1}  e^{-\beta t} \eqqcolon \eps C_1 e^{-\beta t}
    \end{equation}
    for all $t \geq T_1$ and $\eps \in (0,\eps_0)$, where in the second inequality we have applied the finite-time estimate \eqref{eq:finite-time-estimate}. Thus, combining the finite-time estimate \eqref{eq:finite-time-estimate} and the large-time estimate \eqref{eq:large-time-estimate}, we find the desired error bound \eqref{eq:better-error-estimate-ms}. This immediately implies the estimate \eqref{eq:better-error-estimate-ms-uniform}, which completes the proof. \comm{[Note: we are careful here, because \eqref{eq:large-time-estimate} only holds for large time, but we claim \eqref{eq:better-error-estimate-ms} for all time]}
\end{proof}

\begin{rem}
    In contrast to the case without explicit scale separation, we expect that the averaged dynamics \eqref{eq:averaged_dynamics} is a good approximation of the coarse-grained dynamics $\hat{\mu}^\eps_t$. Specifically, the coarse-grained generator converges to the averaged one as $\eps \to 0$ uniformly on finite time intervals, see \cite[Lemma 3.4]{HilderPeletierSharmaTse20}. Moreover, error estimates are also available in this case, see e.g.~\cite[Remark 16.2]{PavliotisStuart08}.
\end{rem}

To complete the proof of Theorem \ref{thm:error-estimate-ms}, it is left to establish an error estimate between the steady states of the effective generator $N^\eps$ and the averaged generator $L^\eps$. In particular, this estimate complements the error estimate for two solutions of the effective equation \eqref{eq:eff} and the averaged equation \eqref{eq:averaged_dynamics} obtained in \cite[Theorem~4.4]{HilderSharma24}. However, we point out that \cite[Theorem~4.4]{HilderSharma24} does not include an estimate between steady states due to the difference in initial data being contained in the bound.

\begin{lem}\label{lem:error-steady-states}
    Let $N^\eps$ and $L^\mathrm{av}$ be given as in \eqref{eq:eff} and \eqref{eq:aver-gen}, respectively, and let $\hat{\rho}^\eps$ and $\pi^\mathrm{av}$ be the corresponding steady states. Then, there exists $C > 0$ and $\eps_0 > 0$ such that
    \begin{equation*} %\label{eq:estimate-equilibria}
        \|\hat{\rho}^\eps - \pi^\mathrm{av}\|_{\mathrm{TV}} \leq C\eps
    \end{equation*}
    for all $\eps \in (0,\eps_0)$.
\end{lem}
\begin{rem}
    We highlight that the result of Lemma \ref{lem:error-steady-states} is not restricted to $|\Y| = 2$ as in the example \eqref{eq:CG-res-gen}. Indeed the proof works for any finite reduced state space $\Y$.
\end{rem}
\begin{proof}
    Since $\hat{\rho}^\eps$ and $\pi^\mathrm{av}$ are steady states, we have that
    \begin{equation*}
        0 = (N^\eps)^T \hat{\rho}^\eps - (L^\mathrm{av})^T \pi^\mathrm{av} = (N^\eps)^T (\hat{\rho}^\eps - \pi^\mathrm{av}) + ((N^\eps)^T - (L^\mathrm{av})^T) \pi^\mathrm{av}
    \end{equation*}
    and therefore,
    \begin{equation}\label{eq:diff-steady-states}
        (N^\eps)^T (\hat{\rho}^\eps - \pi^\mathrm{av}) = - (N^\eps - L^\mathrm{av})^T \pi^\mathrm{av}.
    \end{equation}
    We wish to multiply both sides of~\eqref{eq:diff-steady-states} by $(N^\eps)^{-T}$ on the left, however, $(N^\eps)^T$ is not invertible. Our strategy is to reduce $\R^{|\Y|}$
    to a subspace on which $(N^\eps)^T$ is invertible. With this aim, we first note that
    \begin{equation*}
        (\1, \hat{\rho}^\eps - \pi^\mathrm{av}) = 0, \qquad (\1, (N^\eps - L^\mathrm{av})^T \pi^\mathrm{av}) = ((N^\eps - L^\mathrm{av})\1,\pi^\mathrm{av}) = 0,
    \end{equation*}
    where we use that $\hat{\rho}^\eps, \pi^\mathrm{av} \in \P(\Y)$ and that $N^\eps$ and $L^\mathrm{av}$ are generators. Therefore, both $\hat{\rho}^\eps - \pi^\mathrm{av}$ and $(N^\eps - L^\mathrm{av})^T\pi^\mathrm{av}$ are in $\lrang{\1}^\perp$.

    Since $N^\eps$ is a generator and thus $\mathrm{dim}(\mathrm{ker}(N^\eps)) = 1$, we have that $\mathrm{rank}((N^\eps)^T) = \mathrm{rank}(N^\eps) = |\Y|-1$ due to the rank-nullity theorem 
    and that $(\1, (N^\eps)^T v) = (N^\eps \1, v) = 0$ for all $v \in \R^{|\cY|}$. 
    These observations demonstrate that $\mathrm{ran}((N^\eps)^T) = \lrang{ \1 }^\perp$. Using in addition that $\mathrm{ker}((N^\eps)^T) = \lrang{ \hat{\rho}^\eps }$ (because $N^\eps$ is an irreducible generator and $\hat{\rho}^\eps$ is its unique steady state), we obtain from $\hat{\rho}^\eps \notin \lrang{ \1 }^\perp$ 
    that the restriction $T^\eps \coloneqq (N^\eps)^T |_{\lrang{ \1 }^\perp}$ is invertible as a linear map from $\lrang{ \1 }^\perp$ to itself. 
    Moreover, $(T^\eps)^{-1}$ is bounded uniformly in $\eps \in (0,\eps_0)$. To see this, note that $N^\eps \to L^\mathrm{av}$ as $\eps \to 0$, see~\eqref{eq:quant-estimate-diff-generators}, and that $L^\mathrm{av}$ is also a generator. Since $L^{\av}$ is a generator, the restriction $T^0 \coloneqq (L^{\av})^T \vert_{\lrang{\1}^\perp}$ is also invertible as a map from $\lrang{\1}^\perp$ into itself. Therefore, we have a sequence of invertible maps $T^\eps : \lrang{\1}^\perp \to \lrang{\1}^\perp$, which converge to an invertible map $T^0 : \lrang{\1}^\perp \to \lrang{\1}^\perp$. Therefore, $(T^\eps)^{-1}$ converges to $(T^0)^{-1}$ and is thus uniformly bounded. 
    Taking $(T^\eps)^{-1}$ on both sides of \eqref{eq:diff-steady-states}, we obtain \comm{[1: $T^\eps$ has full rank on $\lrang{ \1 }^\perp$ can be described by a square matrix. 2: $\hat\rho^\eps-\pi^{\av}\in \lrang{\1}^\perp$ for every $\eps$]}
    \begin{equation*} %\label{eq:non-zero-spectrum-est}
        \|\hat{\rho}^\eps - \pi^\mathrm{av}\|
        \leq 
        \|(T^\eps)^{-1}\|
        \|(N^\eps)^T - (L^\mathrm{av})^T\|
        \|\pi^\mathrm{av}\|
        \leq C \eps
    \end{equation*}
    for $\e$ small enough, where we used $\|(N^\eps)^T - (L^\mathrm{av})^T\|  \leq C \eps$, see \eqref{eq:quant-estimate-diff-generators}, in the last estimate. Using that all norms are equivalent on $\R^{|\Y|}$ this completes the proof.
\end{proof}

\subsubsection{Assessing the quality of coarse-graining maps}\label{sec:different-cg-maps}

We now return to the general error estimates without scale separation in Proposition \ref{prop:error-estimate-no-scale-sep} and recall that the constant in the error estimate depends on $\alpha^{-1}$ with $\alpha \leq \min_{y \in \Y} \alpha_{\gLSI}(\rho(\cdot | y), M^y)$. Using the expression for $\alpha$ given in \eqref{eq:alpha}, we can now give an estimate of the quality of the chosen coarse-graining map $\xi$. For this, we note that since the constant in the error estimate behaves like $\alpha^{-1}$, a larger $\alpha$ yields a smaller error and thus indicates a "better" coarse-graining map $\xi$.
While $\alpha$ given in \eqref{eq:alpha} is only a lower bound for the gLSI constants $\alpha_{\gLSI}(\rho(\cdot | y) , L^y)$ and thus only gives a sufficient criterion to evaluate its quality, we also provide an upper bound which is given by $\lambda_y$, the spectral gap of the symmetrised generator corresponding to the pair $L^y$ and $\rho(\cdot | y)$, see \eqref{eq:symGen} and \eqref{eq:upper-bound-gLSI}. Therefore, the spectral gap $\lambda_y$ together with the minimum of the conditional measure $\rho(\cdot | y)$ provide necessary and sufficient criteria to evaluate the quality of the chosen coarse-graining map.

We demonstrate this for the example Markov chain, which is generated by
\begin{equation}\label{eq:generator-example}
    L = \begin{pmatrix}
            -3\eps^{-1} & 2\eps^{-1} & \eps^{-1} & 0 & 0 & 0 \\
            \eps^{-1} & - 1 - 3\eps^{-1} & 2\eps^{-1} & 1 & 0 & 0 \\
            2\eps^{-1} & \eps^{-1} & -2-3\eps^{-1} & 0 & 0 & 2 \\
            0 & 2 & 0 & -2 - 3\eps^{-1} & \eps^{-1} & 2\eps^{-1} \\
            0 & 0 & 0 & 2\eps^{-1} & -3\eps^{-1} & \eps^{-1} \\
            0 & 0 & 1 & \eps^{-1} & 2\eps^{-1} & -1 - 3\eps^{-1}
        \end{pmatrix}
\end{equation}
with $\eps > 0$. This is a generator of a non-reversible Markov chain on the state space $\X = \{1,2,3,4,5,6\}$, and the corresponding network indicating the transitions between the states is depicted in Figure \ref{fig:transition-graph}. Moreover, the steady state is given by
\begin{equation*}
    \rho = \dfrac{1}{42+27\eps} (7+4\eps, 7+6\eps, 7+3\eps, 7+3\eps, 7+5\eps, 7+6\eps)^T
\end{equation*}
and is also $\eps$-dependent. In fact, this generator aligns with the setting in Section \ref{sec:application-CG} and exhibits explicit scale separation. This makes it convenient to evaluate the quality of the coarse-graining map. Since the calculations are extensive, we refrain from including all details here and only present important intermediate results. Instead, we refer to the supplementary material \cite{supplementary-material}, where the full calculations are provided using Mathematica \cite{mathematica}. 

\begin{figure}[ht!]
    \centering
    \includegraphics[width=0.8\linewidth]{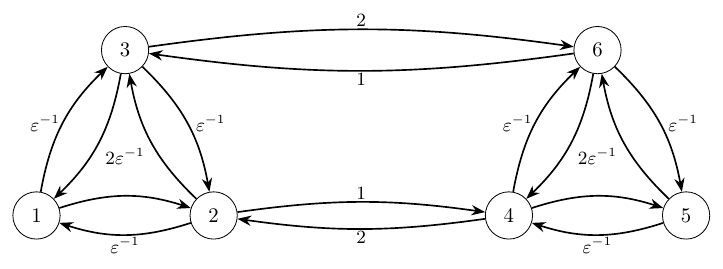}
    \caption{Depiction of the transitions of the Markov chain generated by \eqref{eq:generator-example}.}
    \label{fig:transition-graph}
\end{figure}

The natural choice for the coarse-graining map is given by $\cY = \{a,b\}$ and 
\begin{equation*}
    \xi(x) = \begin{dcases}
        a, & x \in \{1,2,3\}, \\
        b, & x \in \{4,5,6\}.
    \end{dcases}
\end{equation*}
In fact, this is the choice of $\xi$ discussed in Section \ref{sec:error-estimates-with-scale-sep}.
In this case, the restricted generators $L^y$ are given by
\begin{equation*}
    L^a = \dfrac{1}{\eps} \begin{pmatrix}
        -3 & 2 & 1 \\
        1 & -3 & 2 \\
        2 & 1 & -3
    \end{pmatrix}, \qquad 
    L^b = \dfrac{1}{\eps}\begin{pmatrix}
        -3 & 1 & 2 \\
        2 & -3 & 1 \\
        1 & 2 & -3
    \end{pmatrix}
\end{equation*}
and the conditional measures are
\begin{equation*}
    \rho(\cdot | a) = \dfrac{1}{21 + 13\eps} (7+4\eps, 7+6\eps, 7+3\eps)^T, \qquad \rho(\cdot | b) = \dfrac{1}{21 + 14\eps}(7+3\eps, 7+5\eps, 7+6\eps)^T.
\end{equation*}
Notice that the conditional steady states are indeed $\eps$-dependent, which stems from the non-reversibility of the generator $L$. With a direct calculation, we then find that the spectral gaps of the symmetrised generators $L^{a,\rho(\cdot|a)}$ and $L^{b,\rho(\cdot|b)}$ are given by
\begin{equation*}
    \lambda_a = -\frac{9}{2 \eps }+\frac{1}{2}\sqrt{\frac{3}{7}}+\left(-\frac{3}{28}-\frac{199}{196 \sqrt{21}}\right)
   \eps + \mathcal{O}(\eps ^2), \quad \lambda_b = -\frac{9}{2 \eps }+\frac{1}{2}\sqrt{\frac{3}{7}}+\left(-\frac{3}{28}-\frac{179}{196 \sqrt{21}}\right)
   \eps +\mathcal{O}(\eps ^2).
\end{equation*}
and therefore,
\begin{equation*}
    \alpha = \frac{9}{2  (2+\log (3))\eps}+ \mathcal{O}(1) = \mathcal{O}(\eps^{-1}).
\end{equation*}
Thus, restricting to strictly positive initial data satisfying
\begin{equation*}
    \RelEnt_{\eta_0}\left(\dfrac{\hat{\mu}_0}{\eta_0}\right) \leq C \eps
\end{equation*}
the improved estimate \eqref{eq:error-est-non-cg} yields that
\begin{equation*}
    \RelEnt_{\eta_t}\left(\dfrac{\hat{\mu}_t}{\eta_t}\right) \leq C\eps,
\end{equation*}
which matches Theorem \ref{thm:error-estimate-ms}.

Finally, to illustrate a ``bad" choice for the coarse-graining map, let
\begin{equation*}
    \tilde{\xi}(x) \coloneqq \begin{dcases}
        a, & x \in \{1,2,4\}, \\
        b, & x \in \{3,5,6\}.
    \end{dcases}
\end{equation*}
This does not reflect the scale separation of $L$ since it mixes micro-states from different macro-states. Indeed, we find that
\begin{equation*}
    \tilde{L}^a = \begin{pmatrix}
        -2\eps^{-1} & 2\eps^{-1} & 0 \\
         \eps^{-1} & -1-\eps^{-1} & 1 \\
         0 & 2 & -2
    \end{pmatrix} \qquad \tilde{L}^b = \begin{pmatrix}
        -2 & 0 & 2 \\
         0 & -\eps^{-1} & \eps^{-1} \\
         1 & 2\eps^{-1} & -1-2\eps^{-1}
    \end{pmatrix}
\end{equation*}
and proceeding as above yields
\begin{equation*}
    \tilde{\alpha} = \frac{9}{4 (2+\log (3))}+\mathcal{O}(\eps) = \mathcal{O}(1).
\end{equation*}
Thus, as expected, Proposition \ref{prop:error-estimate-no-scale-sep} yields no decay of the error between $\hat{\mu}$ and $\eta$ as $\eps \to 0$.

\begin{rem} \label{r:lamy}
    A key step in the above calculations is the spectral gap of the symmetrised generators $L^{y,\rho(\cdot | y)}$ for $y \in \cY$, which requires solving an eigenvalue problem. However, since the symmetrised generators are self-adjoint matrices on $\R^{|\Lambda_y|}$ equipped with the weighted inner product $(\cdot, \cdot)_{\rho(\cdot | y)}$, the spectral gap can be obtained from a variational problem
    \begin{equation*}
        \lambda_y = \sup_{f \in \langle \1 \rangle^\perp} \dfrac{(f, L^{y,\rho(\cdot | y)} f)_{\rho(\cdot | y)}}{(f,f)_{\rho(\cdot | y)}}.
    \end{equation*}
\end{rem}

\section{Discussion}
\label{sec:discussion}

In this paper, we introduce a generalised version of the classical Poincar\'e and log-Sobolev inequalities for reference measures which are not the steady state of the underlying Markov chain. We prove that the associated constants have similar properties to the classical variants and also depend continuously on the reference measure. We establish a hierarchy between the constants and obtain an explicit lower bound, which is determined by the smallest positive entries of the reference measure and the generator. Moreover, we find an upper and lower bound in terms of the spectral gap of a symmetrised generator. Finally, we apply our framework to establish exponential convergence of two distributions of the same Markov chain and to obtain quantitative error estimates for an effective dynamics in coarse-graining applications.

We now comment on several open questions and interesting directions of future research. 

\paragraph{Assessing quality of coarse-graining maps.} 
Consider a Markov chain $L$ on $\cX$ and a coarse graining map $\xi : \cX \to \cY$ as in Sec 5.3. The corresponding value $\alpha = \alpha(\xi)$ in \eqref{eq:alpha} is a measure for how well the effective dynamics mimic the coarse-grained dynamics, and can therefore be used as a quality score for the choice of $\xi$. This is exemplified in Section~\ref{sec:different-cg-maps}, where the `natural' choice $\xi$ and the `bad' choice $\tilde \xi$ for the coarse graining maps yield $\alpha(\xi) \gg \alpha(\tilde\xi)$.

Beyond providing a diagnostic tool, the quality score $\alpha(\xi)$ induces a principled strategy of constructing coarse-graining maps where $\alpha(\xi)$ can be viewed as an objective function to be maximised. This is in line with the established metastability viewpoint that good macroscopic (or coarse) variables should group states that equilibrate rapidly under the microscopic dynamics, while states separated by slow transitions should remain in distinct coarse clusters. 
This gLSI approach differs from the classical lumpability~\cite{KemenySnell69,Buchholz94} and spectral~\cite{DeuflhardWeber05} approaches. Lumpability identifies a valid coarse-grained chain only under very strong algebraic constraints on the generator -- each microstate in a cluster must have exactly the same transition rate to every other cluster -- a condition that is rarely met in practice. Spectral methods identify metastable sets by analysing the eigenvectors close to the zero eigenvalue of the generator, which encodes the chain's global slow modes and operates entirely at the equilibrium level. In contrast, the gLSI approach, which does not rely on exact algebraic structure or stationarity, instead measures whether the restricted dynamics on each cluster exhibits sufficient \emph{internal contractivity} relative to the chosen reference measure. This provides a criterion for selecting coarse clusters that remain dynamically meaningful even when the system is far from equilibrium. 
Finally, the variational character of the spectral gap for the symmetrised generator  
suggests a natural machine learning perspective wherein one can parametrise a family of coarse-graining maps $\xi_\theta$ and optimise $\theta$ as to maximise $\alpha(\xi_\theta)$. 

\paragraph{Connections to $\Gamma$-calculus.}
A natural question raised by the generalised functional inequalities introduced in this work is whether they admit a geometric characterisation in the spirit of Bakry-\'Emery theory especially beyond reversibility. 
Bakry-\'Emery $\Gamma$-calculus provides a framework which links curvature-dimension conditions of the type $\Gamma_2\geq \alpha\Gamma$, where $\Gamma$ is the carr\'e-du-champ operator and $\Gamma_2$ is the corresponding iterated operator, to functional inequalities for diffusion processes~\cite{BakryGentilLedoux14}. In particular, the curvature condition implies the LSI with constant $\alpha$. 

The picture changes fundamentally for Markov chains due to lack of a \emph{chain rule} which is not available due to the non-local nature of jump processes. More precisely, for Markov chains the usual carr\'e-du-champ operator 
\begin{equation*}
    \Gamma(f,g)(z)=\frac12 \sum_{z'\in Z} M(z,z')\big(f(z')-f(z)\big)\big(g(z')-g(z)\big)
\end{equation*}
is well defined. The first derivative of the entropy satisfies 
\begin{equation*}
    \frac{\dd}{\dd t}\RelEnt_\pi(\varphi_t) = -\sum_{z\in\Z} \pi(z)\Gamma(\varphi_t,\log\varphi_t)(z)
\end{equation*}
where $\varphi_t=\frac{\mu_t}{\pi}$, $\frac d{dt} \mu_t=M^T\mu_t$ and $M^T\pi=0$. However, as opposed to the diffusive case, we cannot control the second derivative of relative entropy purely in terms of the `diffusive' $\Gamma_2$ defined for Markov chains as 
\begin{equation*}
    \Gamma_2(f)=\frac12\big(M\Gamma(f,f)-2\Gamma(f,Mf) \big).
\end{equation*}
Two nonlinear frameworks have been developed to address this issue in the reversible Markov-chain setting. 
The entropic–transport approach of~\cite{erbarMaas12,FathiMaas16} defines curvature via convexity of entropy along discrete transport geodesics, while discrete state-space $\Gamma$-calculus~\cite{CaputoDaiPraPosta09,WeberZacher21} introduces an entropy-adapted carré-du-champ $\Gamma_{\Upsilon}$ leading to the exact entropy identities and the LSI. Both these approaches are intrinsically nonlinear and rely on reversibility and stationarity.

In view of this it is natural to ask whether the generalised PI and LSI considered in this work can be related to a curvature notion that remain meaningful for non-reversible Markov chains and possibly with a non steady-state reference measure. 
% \US{Need to add a few more lines.}

\paragraph{Countably infinite state space and coarse-grained diffusion processes.} 
Throughout the article, we restrict to continuous-time Markov chains on finite state spaces. However, we expect that the definitions of the generalised functionals, see \eqref{def:genDirichlet} and \eqref{def:genFisher}, and the generalised Poincar\'e and log-Sobolev inequalities, see \eqref{def:genPI} and \eqref{def:genLSI}, generalise to the case of a countably infinite state space, cf.~\cite{HilderPeletierSharmaTse20}, where an alternative generalisation of the Fisher information (see Remark \ref{rem:alternative-generalistion}) was introduced on countable state spaces. Moreover, we expect that the properties of the generalised functionals in Sections \ref{sec:Gen}, and some properties of the gPI and gLSI constants, in particular the continuity, see Theorem \ref{thm:continuity-alpha-LSI}, generalise under appropriate additional assumptions, such as assuming that $M$ has uniformly bounded rates, which guarantees that the sequence $(M(z,z'))_{z' \in \Z}$ is integrable for any $z \in \Z$. However, our current proof of the lower bounds presented in Section \ref{s:al:gLSI:LB} relies heavily on the fact that the reference measure is strictly positive, i.e.~$\zeta_\ast > 0$, which fails in countably infinite state spaces. Therefore, under what conditions a lower bound can still be obtained in the countably infinite setting is still open. Similarly, the generalisation of the error estimates for clustering of Markov chains (see Section \ref{sec:application-CG}) to infinite state spaces is an open question, and we refer to \cite[Section 6]{HilderSharma24} for a discussion.

We now turn to Markov processes on a continuous state space, specifically the coarse-graining of diffusion processes in molecular dynamics. Here, coarse-graining estimates similar to Section~\ref{sec:application-CG} have been developed for both reversible~\cite{LegollLelievreSharma18,LegollLelievreOlla17,ZhangHartmannSchutte16,LelievreZhang18} and non-reversible setting~\cite{DLPSS18,LegollLelievreSharma18,HartmannNeureitherSharma20}, under the key assumption that the conditional steady-state of the process, denoted by $\rho(\cdot|\xi(z)=y)$, satisfies either a LSI or a PI inequality with respect to a coarse-grained generator $L^y$ -- the notation is as in Section~\ref{sec:application-CG} to indicate similarities. In the reversible setting it turns out that $(L^y)^*\rho(\cdot|\xi(z)=y)=0$, where $(L^y)^*$ is the adjoint operator, and therefore the LSI inequality follows under classical growth conditions on the coefficients~\cite{CGAW09,BakryGentilLedoux14}. However, this relationship is not clear in the setting of non-reversible diffusions discussed in~\cite{LegollLelievreSharma18,HartmannNeureitherSharma20}. 

In light of these observations, it is natural to ask whether the non-equilibrium principles developed in this paper can be extended to diffusion processes. Several natural questions follow from here. For instance, how can one formulate practical lower bounds for the diffusion version of the gLSI constant (note that in the diffusion setting the entropy-dissipation is the same for both the reversible and non-reversible setting), and do analogous spectral lower bounds as in Theorem~\ref{thm:pos} carry over? It is also natural to wonder about the scaling behaviour of these gPI, gLSI inequalities in the presence of explicit scale-separation which connects coarse-graining to classical averaging problems~\cite{HartmannNeureitherSharma20}. Finally, can the gLSI inequality (or a suitable lower bound) be used as a diagnostic tool to compare coarse-graining maps? This is especially relevant in the field of molecular dynamics.

\section*{Acknowledgements}
PvM was supported by JSPS KAKENHI, grant Numbers JP20K14358 and JP24K06843.
BH was partially supported by the Swedish Research Council -- grant no.~2020-00440 -- and the Deutsche Forschungsgemeinschaft (DFG, German Research Foundation) -- Project-IDs 444753754 and 543917644.

\section*{Data availability statement}

The supplementary material providing the detailed calculations for the example in Section \ref{sec:different-cg-maps} is available at \url{https://github.com/Bastian-Hilder/non-equilibrium-functional-ineq}, see also \cite{supplementary-material}. 

\appendix

\section{Auxiliary results}\label{app:failure-classical-inequalities}

The following result shows that the classical Poincar\'e inequality with constant
\[
  \tilde \alpha_{\PI}(\zeta)
  = \inf_{\substack{f\in\R^{|\Z|} \\ f\notin \langle\1\rangle}} \frac{ (f,-Mf)_\zeta}{\Var_\zeta(f)},
\]
see \eqref{def:class-PI-con-genNaive},
only makes sense when $\zeta$ is the steady state. 

\begin{prop}\label{prop:cFI-Fail}
    Let $\zeta\in\P_+(\Z)$. We have $\tilde \alpha_{\PI}(\zeta)>0$ if and only if $\zeta=\pi$. Moreover, if $\zeta\neq \pi$, then $\tilde \alpha_{\PI}(\zeta) = -\infty$.  
\end{prop}
\begin{proof}
Take any $f \in \R^{|\Z|}$ with $f\notin \langle\1\rangle$ and use \eqref{eq:Rz-directSum} to decompose it as $f=c\1+h$ with $c\in \R$ and $h\in \langle\1\rangle^{\perp_\zeta}$ with $h \neq 0$. Then,
\begin{equation*}
    (f,-Mf)_\zeta = (h,-Mh)_\zeta + c(\1,-Mh)_\zeta, \ \ \Var_\zeta(f) = (h^2,\1)_\zeta=(h,h)_\zeta, 
\end{equation*}
which leads to 
\begin{align*}
    \tilde \alpha_{\PI}(\zeta) 
    = \inf_{c\in \R} 
    \inf_{\substack{ h\in \langle\1\rangle^{\perp_\zeta} \\ h \neq 0 }}\frac{(h,-Mh)_\zeta + c(1,-Mh)_\zeta}{(h,h)_\zeta}.
\end{align*}
Clearly $\tilde \alpha_{\PI}(\zeta)=-\infty$ unless $(1,-Mh)_\zeta=0$ for every $h\in \langle\1\rangle^{\perp_\zeta}$. We claim that the latter is impossible unless $\zeta = \pi$. If this claim holds, then the result follows, since $\tilde \alpha_{\PI}(\pi) = \alpha_{\PI}(\pi) >0$.

It is left to prove the claim that $\zeta = \pi$ if and only if $(1,-Mh)_\zeta=0$ for every $h\in \langle\1\rangle^{\perp_\zeta}$. Note that $(1,-Mh)_\zeta=0$ for every $h\in \langle\1\rangle^{\perp_\zeta}$ if and only if $\zeta$ is an eigenvector of $M^T$. Since $M$ is an irreducible generator, the Perron-Frobenius theorem applies, and states that the only eigenvector of $M^T$ with strictly positive entries is the steady state $\pi$.   
\end{proof}

\begin{rem}
    While we lack a similar general result for the classical log-Sobolev inequality, it is still straightforward to find examples where the naive generalisation of the classical log-Sobolev constant is negative for any measure $\zeta$ other than the steady state $\pi$.

    In fact, already for a two-point state space $\Z = \{1,2\}$ such an example exists. Indeed, take the generator
    \begin{equation*}
        M = \begin{pmatrix}
            -1 & 1 \\ a & -a
        \end{pmatrix}
    \end{equation*}
    for an arbitrary $a > 0$, and take $\zeta = (\beta, 1- \beta)$ for any $\beta \in (0,1)$ such that $\zeta \neq \pi = (\tfrac{a}{1+a}, \tfrac{1}{1+a})$. Then, for any density $\varphi \in \cD_\zeta(\Z)$ we calculate the naive generalisation of the relative Fisher information, see \eqref{def:Fisher}:
    \begin{equation*}
        (\varphi, -M \log \varphi)_\zeta 
        = \log\left(\dfrac{\varphi_1}{\varphi_2}\right) (\varphi_1 \beta - \varphi_2 (1-\beta) a) 
        = \varphi_2 \log\left(\dfrac{\varphi_1}{\varphi_2}\right) \left(\dfrac{\varphi_1}{\varphi_2} \beta - (1-\beta) a\right) =: \varphi_2 F\left(\dfrac{\varphi_1}{\varphi_2}\right).
    \end{equation*}
    A straightforward extreme value analysis of $F$ then shows that there exists a $\gamma > 0$ such that $F(\gamma) < 0$ (we use here $\beta \neq \tfrac{a}{1+a}$, that is, $\zeta \neq \pi$). Recalling that the relative entropy $\RelEnt_\zeta$ is non-negative, this shows that the naive generalisation of the classical log-Sobolev inequality constant is generically negative.
\end{rem}

%Bibliography 
\bibliographystyle{alphainitials}
\bibliography{BibArXiv.bib}

\end{document}